\documentclass[a4paper, 12pt, leqno]{article}

\usepackage{amsmath}
\usepackage{amssymb}
\usepackage{ascmac}
\usepackage{amsthm}
\usepackage{appendix}

\usepackage{bbm}
\usepackage{color}
\usepackage{enumitem}
\usepackage{mathrsfs}
\usepackage[pdftex]{graphicx}
\usepackage{authblk}
\usepackage{float}
\usepackage{wrapfig}
\usepackage{layout}
\usepackage{cite}
\usepackage{braket}
\usepackage
[draft=false,setpagesize=false,pdfstartview=FitH,
colorlinks=true,linkcolor=blue,citecolor=magenta,pagebackref=true,bookmarks=false]
{hyperref}
\usepackage{mathtools}
\usepackage{mleftright}
\numberwithin{equation}{section}
\usepackage{physics2}
\usephysicsmodule{ab}
\usephysicsmodule{diagmat}
\usephysicsmodule{xmat}
\usepackage{derivative}
\usepackage{autobreak}

\newcommand{\R}{\mathbf{R}}

\newcommand{\RN}{\mathbf{R}^N}

\newcommand{\N}{\mathbf{N}}
\newcommand{\e}{\varepsilon}

\newcommand{\cG}{\mathcal{G}}

\newcommand{\cD}{\mathcal{D}}

\newcommand{\scrG}{\mathscr{G}}

\def\ov#1{\overline{#1}}
\def\wh#1{\widehat{#1}}
\def\wt#1{\widetilde{#1}}

\newcommand{\rad}{\mathrm{rad}}

\DeclareMathOperator{\diver}{div}
\DeclareMathOperator{\supp}{supp}

\DeclareMathOperator{\dist}{dist}

\def\s{\sigma}

\usepackage{todonotes}
\usepackage[truedimen, margin=20truemm, bottom=20truemm, footskip=15truemm]{geometry}
\usepackage{layout}
\usepackage{lineno}
\usepackage{bm} 

\usepackage{aliascnt}

\usepackage[capitalise,noabbrev,nameinlink]{cleveref}

\AtBeginEnvironment{appendices}{\crefalias{section}{appendix}}

\newtheorem{theorem}{Theorem}[section]
\crefname{theorem}{Theorem}{Theorems}

\newaliascnt{lemma}{theorem}
\newtheorem{lemma}[lemma]{Lemma}
\aliascntresetthe{lemma}
\crefname{lemma}{Lemma}{Lemmas}

\newaliascnt{proposition}{theorem}
\newtheorem{proposition}[proposition]{Proposition}
\aliascntresetthe{proposition}
\crefname{proposition}{Proposition}{Propositions}

\newaliascnt{corollary}{theorem}

\aliascntresetthe{corollary}
\crefname{corollary}{Corollary}{Corollaries}

\theoremstyle{definition}
\newaliascnt{definition}{theorem}

\aliascntresetthe{definition}
\crefname{definition}{Definition}{Definitions}

\theoremstyle{remark}
\newaliascnt{remark}{theorem}
\newtheorem{remark}[remark]{Remark}
\aliascntresetthe{remark}
\crefname{remark}{Remark}{Remarks}
 
\usepackage{authblk} 

\makeatletter
\renewenvironment{proof}[1][\proofname]{\par
	\pushQED{\qed}%
	\normalfont \topsep6\p@\@plus6\p@\relax
	\trivlist
	\item\relax
	{\bfseries
		#1\@addpunct{.}}\hspace\labelsep\ignorespaces
}{%
	\popQED\endtrivlist\@endpefalse
}
\makeatother

\providecommand{\keywords}[1]
{\small	\textit{Keywords and phrases:} $(2,q)$-Laplacian; $L^2$–normalized solution; mass critical exponents; existence results; Liouville theorems.}

\providecommand{\MSC}[1]
{\small	2020 \textit{Mathematics Subject Classification:} 35J20, 35J60, 35B38, 35A15}

\begin{document}

\title{Normalized Solutions for the $(2,q)$-Laplacian Operator Between Mass-Critical Exponents}

\author[1,2]{Laura Baldelli}

\author[3]{Norihisa Ikoma}

\affil[1]{\small IMAG, Departamento de An\`alisis Matem\`atico, Universidad de Granada, \par
\centering Campus Fuentenueva, 18071 Granada, Spain
}
\affil[2]{\small 
Institute for Analysis, Karlsruhe Institute of Technology (KIT)\par
\centering D-76128 Karlsruhe,
Germany\par 
 \texttt{ laura.baldelli@kit.edu }}
\affil[3]{\small Department of Mathematics, Faculty of Science and Technology, Keio University\par
\centering
Yagami Campus, 3-14-1 Hiyoshi, Kohoku-ku, Yokohama, Kanagawa 223-8522, Japan\par 
\texttt{ikoma@math.keio.ac.jp} }

\maketitle

\begin{abstract}
This paper concerns the existence of normalized solutions to a class of $(2,q)$-Laplacian equations 
with a power type nonlinearity in the intermediate regime between the two mass critical exponents $2(1+2/N)$, $q(1+2/N)$. 
More precisely, we prove the existence of solutions with negative energy obtained through a global minimization procedure, and of solutions with positive energy established via a local minimization technique and a mountain-pass argument. Furthermore, we derive both existence and nonexistence results for the zero-mass case $\lambda = 0$, highlighting the role of the mixed diffusion in determining the qualitative behavior of solutions. Specifically, this paper's novelty lies in providing a comprehensive understanding of the intermediate cases that arise when the non-homogeneous $(2,q)$-Laplacian operator appears. Our analysis combines variational methods, compactness arguments, and delicate energy estimates adapted to the nonhomogeneous nature of the $(2,q)$-Laplacian operator.

\end{abstract}


\MSC{}

\section{Introduction}

The present paper concerns the existence of solutions $(\lambda,u)$ to the following $(2,q)$-Laplacian equation
\begin{equation}\label{eq_main}
	-\Delta u - \Delta_q u + \lambda u = \alpha \vab{u}^{p-2} u \quad \text{in} \ \RN, 
\end{equation}
under the constraint
\begin{equation}\label{eq_mass}
  \Vab{u}_2^2 \coloneq \int_{{\mathbb{R}^N}} {{|u|}^2} \odif{x}=m,
\end{equation}
where $N \geq 1$, $q>2$, $m>0$ $\alpha\ge 0$, $\lambda\in\R$, $\Delta_q u=\diver ({| \nabla u |^{q - 2}}\nabla u )$ is the $q$-Laplacian of $u$,  and the exponent $p$ of the nonlinearity is between the two mass critical exponents, namely
\begin{equation}\label{middle}
p_2 \coloneq 2 + \frac{4}{N} < p < q \ab( 1 + \frac{2}{N} ) \eqcolon p_q.
\end{equation}
Observe that $p_q<q^* \coloneq Nq / (N-q)_+$ since $2<q$. 
In particular, we are seeking normalized solutions to \eqref{eq_main}, since \eqref{eq_mass} imposes a normalization on its $L^2$-mass, which can be obtained by searching critical points of the following functional 
\begin{equation}\label{defE}
	E_\alpha(u) \coloneq \frac{1}{2} \Vab{\nabla u}_2^2 + \frac{1}{q} \Vab{\nabla u}^q_q - \frac{\alpha}{p} \Vab{u}_p^p
\end{equation}
restricted to
\begin{equation}\label{defSm}
S_m \coloneq \Set{u \in X | \Vab{u}_2^2 = m},
 \end{equation}
where 
\[
X\coloneq \Set{ u \in H^1(\RN) | \,\,|\nabla u| \in L^q(\RN) }, \quad 
X_\rad \coloneq \Set{u \in X | u(x) =u(\vab{x})};
\]
see \cref{s:pre} for details. Note that $\lambda$ appears as Lagrange multipliers and that $\lambda$ is part of the unknown.


The more general $(p,q)$-Laplacian operator $\Delta_p + \Delta_q$ 
appears in the models in nonlinear elasticity (\cite{Zh86}) and solitary waves for elementary particles (\cite{gDeK,BDAFP00}). 
We also refer to \cite{BFMP99,BFP98} for related problems. 
Moreover, the $(2,4)$-Laplacian operator and its extensions appear as an approximation of the Born–Infeld operator, defined by
\[
\mathcal Q(u) := -{\rm div}\left(\frac{\nabla u}{\sqrt{1 - |\nabla u|^2}}\right).
\]
The electromagnetic theory introduced by Born and Infeld (see \cite{Bnat, B, BInat, BI}) provides a nonlinear alternative to classical Maxwell theory. 
Its importance lies in offering a unified framework for electrodynamics and, notably, in providing a satisfactory resolution to the well-known \emph{infinite-energy problem}: 
indeed, in the Born–Infeld model, the electromagnetic field generated by a point charge possesses finite energy.
By performing the Taylor expansion of $1/\sqrt{1 - |s|}$ up to order $k$, one obtains the approximate operator
\[
\mathcal Q(u) \sim -\Delta u - \frac{1}{2}\Delta_4 u - \frac{3}{4 \cdot 2}\Delta_6 u - \dots - \frac{(2k-3)!!}{(2k-1)!!}\Delta_{2k} u;
\]
see \cite{BodAPo16,Ki12,PoWa18}.

\medskip

In this paper, motivated by the fact that physicists are often interested in normalized solutions, since the prescribed mass naturally arises in nonlinear optics and in the theory of Bose–Einstein condensates (see \cite{fra, ma} and the references therein), we look for solutions of \eqref{eq_main} in $X$ having a prescribed $L^2$-norm, as follows from \eqref{eq_mass}.

In the literature, the existence of normalized solutions $(\lambda, u)$ to the semilinear elliptic equation 
involving the Laplace operator has been extensively investigated in recent years. 
In the $L^2$-subcritical case, namely $p < p_2$, the functional restricted to the constraint is coercive, 
so a global minimizer can be obtained by minimizing on the sphere (see \cite{CaLi82,Lions84II, s82}). 
In contrast, this approach fails in the other cases: for instance, in the $L^2$-supercritical case ($p > p_2$), 
the functional restricted to the sphere is no longer bounded from below. 
In \cite{Je97}, Jeanjean treated the $L^2$-supercritical case by introducing a mountain-pass structure for an auxiliary functional. 
We also refer to \cite{TSdV,bm21,CiGaIkTaI,CiGaIkTaII,HiTa19,IkTa19,JeLu20,JeLu22,JeZhZh24,MeSc24,Sc22,Sh14,So20a,So20b} for generalizations 
of $L^2$-subcritical, $L^2$-critical and $L^2$-supercritical cases. 
For a counterpart of the $p$-Laplacian, we refer to \cite{LoZhZh24, LoZh24, ShWa24,WaLiZhLi20,ZhLeLe24, ZhZh22}.


Moving to the $(2,q)$-Laplacian operator, the first contribution is due to \cite{BaYa25}, 
where the authors studied the existence of normalized solutions of \eqref{eq_main} with $\alpha = 1$, 
for $p$ below and above both mass-critical exponents $p_2$ and $p_q$, as well as in the mass-critical cases. 
A generalization in terms of the operator was later given in \cite{BaMePo25}, together with an application to Born–Infeld theory, 
and in \cite{CaRa24}, where the $(p,q)$-Laplacian operator is considered together with a more general nonlinearity. 
Notice that the papers mentioned above do not treat the case \eqref{middle}.

To the best of the authors' knowledge, papers dealing with the case \eqref{middle} are \cite{DiJiPu25,HuLuWa25}. 
In \cite{DiJiPu25}, under $1<q<2 = N$ or $N \geq 3$ and $2<q<N$, 
instead of $\vab{u}^{p-2} u$, the authors considered a generalized nonlinearity $f=f(u)$ with $f_+(u) / u \to 0$ as $u \to 0$ and 
$f(u) / \vab{u}^{\ov{p}-1} \to 0$ as $\vab{u} \to \infty$, where $f_+(u) \coloneq \max \{0,f(u)\}$ and $ \overline{p} \coloneq  (1+2/N) \max\{2,q\} = \max\{p_2,p_q\}$. 
In particular, exploiting an idea introduced in \cite{bm21,MeSc24}, they obtained the existence of global minimizers 
through the following minimizing problem $\inf_{D(m)} E_\alpha$ where $D(m) \coloneq \set{ u \in X | \Vab{u}_2^2 \leq m}$. 
It is worth noting that when $N=2$, by $1<q<2$, $p_q<p_2=\overline{p}$ holds and the situation is different from our case. 
On the other hand, in \cite{HuLuWa25}, an inhomogeneous Hardy-type nonlinearity is treated under $q > 2$ and $N \geq 1$. 
Nevertheless, their analysis focuses only on minimizers, and their proof of existence differs from ours. 
In fact, due to the weight $\vab{x}^{-b}$ of $\vab{u}^{p-2}u$, the equation is not invariant by translations and 
the lack of the compactness does not occur in \cite{HuLuWa25}. 
We also refer to \cite{ZhZhLi25}, where the authors study \eqref{eq_main} with $\alpha = 1$, fixing the $L^p$-norm instead of \eqref{eq_mass}, 
and obtain a result analogous to that of \cite{HuLuWa25} in the case \eqref{middle}.

Inspired by \cite{BaYa25}, where the subcritical and supercritical cases with respect to both mass-critical exponents $p_2$ and $p_q$ are analyzed, 
the aim of this paper is to investigate the intermediate case, namely \eqref{middle}. 
In particular, we establish the existence of solutions with negative energy, obtained through a global minimization process, and solutions with positive energy, 
derived via a local minimization technique and a mountain-pass argument. 
Moreover, we provide existence and nonexistence results for the zero-mass case, i.e., $\lambda=0$. 
As stated in the above, as far as we know, 
equation \eqref{eq_main} with a power-type nonlinearity between the two mass-critical exponents has not previously been studied in all its facets. 
The objective of this paper is to address this gap. We also note that in the Appendix, we prove the well-known Pohozaev identity and regularity as a necessary preliminary step. 
In particular, we prove the $C^{2,\gamma}_{loc}$ regularity for every weak solution to our problem, a result which is often taken for granted or not exhaustively proven in other papers.

In what follows, we present the main results along the directions outlined above.

\subsubsection*{Main Results}


We first consider the global minimization problem
\[
\displaystyle e_\alpha (m) \coloneq \inf_{u \in S_m} E_\alpha(u).
\]

\begin{theorem}\label{theomin}
	Let $N \geq 1 $, $q > 2$ and $p_2 < p < p_q $. Then there exists $\alpha_0(m) > 0$ such that  
	\begin{enumerate}[label={\rm (\roman*)}]
		\item 
		when $\alpha > \alpha_0(m)$, $e_\alpha(m) < 0 $ holds and any minimizing sequence for $e_\alpha(m)$ is relatively compact in $X$ up to translations; 
		in particular, $e_\alpha(m)$ is achieved by some $u\in S_m$;
		
		\item 
		when $\alpha = \alpha_0(m)$, $e_{\alpha_0(m)} (m) = 0$ and $e_{\alpha_0(m)} (m)$ is achieved by some $u \in S_m$; 
		
		\item 
		when $0< \alpha < \alpha_0(m)$, $e_\alpha(m) = 0$ and $e_{\alpha} (m)$ is never attained. 
	\end{enumerate}
	Furthermore, any minimizer $u$ corresponding to $e_\alpha(m)$ is a 
	solution of \eqref{eq_main} with a strictly positive Lagrange multiplier $\lambda$, 
	and either $u>0$ in $\R^N$ or $u<0$ in $\R^N$. 
	Finally, there exists a radially symmetric minimizer corresponding to $e_\alpha(m)$.  
\end{theorem}

\begin{remark}
	(i) 
	As in \cite{DiJiPu25}, instead of $\alpha$, we may use $m$ as a parameter and obtain a similar result to \cref{theomin} depending on the size of $m$. 
	See \cref{Rem:alp0m} in this perspective. 
   
  (ii) By \cref{l:reg}, each minimizing $u$ corresponding to $e_\alpha(m)$ is of class $C^2$. 
  Thus, with the aid of \cite[Theorem 2]{Ma09}, 
  when $N \geq 2$, as in \cite[the proof of Theorem 1.4]{JeLu22CVPDE}, we may prove that $u$ is radially symmetric up to translations and monotone. 
  Similarly, when $N=1$, if $\vab{u(x_0)} = \max_{\R} \vab{u}$, then $u'(x_0) = 0$ and $u$ is a solution of 
  \begin{equation}\label{eq:ODE}
  - \ab( 1 + (q-1) \vab{u'}^{q-2} ) u'' + \lambda u = \alpha \vab{u}^{p-2} u  \quad \text{in} \ \R.
  \end{equation}
  Since $v(x) \coloneq u(x)$ if $x \geq x_0$ and $v(x) \coloneq u(2x_0 - x )$ if $x < x_0$ is also a solution of \eqref{eq:ODE}, 
  when $q \geq 3$, the unique solvability applies at $x=x_0$ to deduce $u(x) = u(2x_0-x)$ for all $x \geq x_0$ and $u$ is symmetric. 
  From 
  \[
  0 = \frac{1}{2} \ab(u'(x))^2 + \frac{q-1}{q} \vab{u'(x)}^q + \frac{\alpha}{p} \vab{u(x)}^p - \frac{\lambda}{2} u^2(x) \quad \text{for any $x \in \R$}, \quad 
  \frac{\lambda}{2} u(x_0)^2 = \frac{\alpha}{p} \vab{u(x_0)}^p,
  \]
  it is easily seen that $u''(x_0) \neq 0$ and $\vab{u'(x)} > 0$ for every $x \in (x_0,\infty)$, which implies that 
  $u$ is monotone when $N=1$ and $q \geq 3$. 
\end{remark}

We first state differences between \cite{DiJiPu25} and \cref{theomin}. 
As explained above, when $N=2$, the authors in \cite{DiJiPu25} studied the case which is different from \eqref{middle}. 
On the other hand, \cref{theomin} deals with the case $q \in (2,\infty)$ for all $N \geq 1$. 
Another difference is that in \cite{DiJiPu25}, the auxiliary minimizing problem is used, 
while we directly investigate the minimizing problem $e_\alpha(m) = \inf_{S_m} E_\alpha$. 
Furthermore, the existence of minimizers is obtained for $\alpha = \alpha_0(m)$, 
which is not discussed in \cite{DiJiPu25}.

The strategy for proving \cref{theomin} is quite standard and relies on analyzing the sign of $e_\alpha(m)$ according to the size of $\alpha$. 
In particular, we will find an $\alpha_0(m) > 0$ such that $e_\alpha(m)$ is negative for $\alpha > \alpha_0(m)$, while $e_\alpha(m) = 0$ for $\alpha \le \alpha_0(m)$, 
with the value $e_{\alpha_0(m)}(m)$ being attained, unlike the case $\alpha < \alpha_0(m)$. 
To achieve this goal, an auxiliary minimizing problem (see \eqref{mp}) is introduced. 
In particular, in \cref{prop:min}, we obtain the existence of minimizers to \eqref{mp} including the case $N=1$, 
which yields \cref{theomin} (ii). 
Then, by exploiting some properties of $e_\alpha(m)$, such as subadditivity on $m$, 
and starting from any minimizing sequence, we are able to rule out vanishing and dichotomy, thus obtaining compactness up to translations. 
Moreover, by using the results in \cref{App}, 
we deduce the sign properties of any minimizer and the associated Lagrange multiplier.

Next, we move to the existence of another type of critical points, that is, local minimizers and mountain pass type critical points, 
which were not studied in \cite{DiJiPu25,HuLuWa25} under \eqref{middle}. 
Inspired by \cite[Theorem 1.2]{JeLu22}, we first study a local minimization problem whose result can be summarized as follows.

	\begin{theorem}\label{theorem:local}
		Let $N \geq 1$, $q > 2$, $p_2 < p < p_q $ and $\alpha_0=\alpha_0(m)>0$ be the number in \cref{theomin}. 
		Then there exists $\rho_0 = \rho_0(\alpha_0, m) > 0$ such that for any $\alpha \in (0, \alpha_0]$, the following statements hold 
		for a local minimizing problem: 
		\begin{equation*}
			\ov{e}_\alpha(m) \coloneq \inf_{u \in S^{\rho_0}_m} E_\alpha (u) \geq e_\alpha(m) = 0,
		\end{equation*}
		where 
		\begin{equation}\label{def-Smrho}
			S^{\rho_0}_m := \Set{ u \in S_m | K(u) \coloneq \Vab{\nabla u}_2^2 + \Vab{\nabla u}_q^q > \frac{ \rho_0(\alpha_0, m) }{2} };
		\end{equation}
		\begin{enumerate}[label={\rm (\roman*)}]
			\item 
			there exists $\alpha_0' \in (0,\alpha_0)$ such that when $\alpha \in (\alpha_0', \alpha_0]$, 
			$\ov{e}_\alpha(m)$ is achieved by some $v \in S^{\rho_0}_m$ with 
			\begin{equation*}
				E_\alpha (v)=
				\begin{dcases}
					\overline{e}_\alpha(m) > 0 &\text{for} ~ \alpha \in (\alpha_0',\alpha_0),\\
					\overline{e}_\alpha(m) = 0 = e_\alpha(m) & \text{for} ~ \alpha = \alpha_0;
				\end{dcases}
			\end{equation*}
			
			\item 
			the local minimizer $v \in S^{\rho_0}_m$ can be chosen so that 
			$v$ has a radial symmetry and a constant sign on $\R^N$, is a solution of \eqref{eq_main} with the associated Lagrange multiplier being positive, 
			and is an energy ground state solution, that is, 
			\[
			\ov{e}_\alpha(m) = E_\alpha(v) = \inf \Set{ E_\alpha(u) | \text{$u \in S_m$ is a critical point of $E_\alpha |_{S_m}$} };
			\]
			
			\item 
			the map $(\alpha_0',\alpha_0] \ni \alpha \mapsto \ov{e}_\alpha(m)$ is strictly decreasing.
		\end{enumerate}
\end{theorem}

\begin{remark}
	Let us point out how the assumptions on the nonlinearity considered in \cite{JeLu22} should be modified in our setting. 
	In fact, conditions $(f_1)$, $(f_3)$, $(f_4)$, and $(f_5)$ remain unchanged. 
	On the other hand, in $(f_2)$ the denominator of the first limit should involve $q^* - 1$ instead of $2^* - 1$, 
	while in the last one $p_q$ should appear in place of $p_2$. 
	These adjustments seem to be the natural modifications in our framework.
\end{remark}

Let us emphasize that \cref{theorem:local} is not a mere adaptation of the result in \cite{JeLu22} 
due to the presence of the $(2,q)$-Laplacian operator, instead of the standard Laplacian. 
One of differences occurs when we consider a splitting of minimizing sequences $(u_n)_{n \in \N}$ to $\overline{e}_\alpha(m)$ used in \cite{JeLu20,JeLu22}. 
To obtain the compactness of $(u_n)_{n \in \N}$, it is important to obtain the splitting of the form $u_n = u_\infty + w_n$ 
with $E_\alpha(u_n) = E_\alpha(u_\infty) + E_\alpha(w_n) + o_n(1)$ where $u_\infty$ is the weak limit of $(u_n)_{n \in \N}$. 
For any minimizing sequence $(u_n)_{n \in \N}$, since we only know the weak convergence of $(\nabla u_n)_{n \in \N}$ in $L^q(\R^N)$, 
it is not clear whether $E_\alpha(u_n) = E_\alpha (u_\infty) + E_\alpha(w_n) + o_n(1)$ holds. 
To overcome this difficulty, in a similar way to \cite{DiJiPu25,Ik14}, 
we prove that a Palais--Smale sequence can be obtained from any minimizing sequence with the same asymptotic behavior by Ekeland's principle. 
This together with \cite[Theorem 2.1]{BoMu92} and the Brezis--Lieb lemma \cite{BrLi83} enables us to get the desired splitting. 
The details can be found in the proof of \cref{p:ex-loc-min}. 
Another difference comes from \eqref{middle}. 
Indeed, in our case, if $u \in S^{\rho_0}_{m'}$ with $m' \in (0,m]$ satisfies $E_\alpha(u) = \overline{e}_\alpha(m)$, 
then it is not immediate to deduce from the Pohozaev identity that the associated Lagrange multiplier to $u$ is strictly positive due to \eqref{middle} 
(cf. \cite[Proof of Lemma 3.6]{JeLu22}). 
Though it is slight, modifications are necessary; see \cref{l:no-critical,l:pos-Lag} and the proof of \cref{p:ex-loc-min}.

We observe that in \cite{JeLu22} the mass $m$ was the varying parameter, whereas in the present setting this role is played by $\alpha$, and the mass may now take any positive value. Hence, the dependence on $m$ in \cite{JeLu22} is here replaced by a dependence on $\alpha$, while $m>0$ remains completely unrestricted. For this reason, one may expect that the roles of $\alpha$ and $m$ could be interchanged in order to reformulate the result, in analogy with \cref{theomin}. However, we have chosen this framework as it provides a clearer perspective on the problem.
Moreover, the existence of a local minimizer, its qualitative properties, and the positivity of the associated Lagrange multiplier are obtained only after a careful analysis, since (locally) uniform assertions in $m$ and $\alpha$ are necessary in the proofs.

Next, we study the existence of positive critical points of mountain pass type by following the approach in \cite{Je97,JeLu22}. 
In this case as well, the mixed regime \eqref{middle} alters the structure of the problem and, in particular, affects the behavior of the associated auxiliary functional. Therefore, a delicate analysis is required to overcome the additional difficulties arising in this different framework. Indeed, a Liouville type assumption appear in the statement of the following theorem.

\begin{theorem}\label{t:ex-mp}
	Let $N \geq 1$, $2 < q < \infty$, $p_2 < p < p_q$, $m>0$ and $\alpha \geq \alpha_0'$, where $\alpha_0'(m)$ is in \cref{theorem:local}. 
	In addition, suppose 
	\begin{enumerate}[label*={\rm (A)}, ref=A]
		\item \label{A}
		there exists no nontrival solution to 
		\[
		-\Delta u - \Delta_q u = \alpha \vab{u}^{p-2} u \quad \text{in} \ \RN, \quad \Vab{u}_2^2 \leq m, \quad u \in X_\rad.
		\]
	\end{enumerate}
	Then there exists $u_\infty\in X$ a solution to \eqref{eq_main} such that $u_\infty > 0$ in $\RN$ and 
	\[
	E_\alpha (u_\infty) 
	> 
	\begin{dcases}
		\ov{e}_\alpha (m) & \text{if $ \alpha_0' < \alpha \leq \alpha_0$}, \\
		0 > e_\alpha(m) & \text{if $\alpha > \alpha_0$}. 
	\end{dcases}
	\]
\end{theorem}

We also give a sufficient condition to ensure \eqref{A}. 
To this end, let us consider 
\begin{equation}\label{eq_zero_mass}
	-\Delta u - \Delta_q u = \alpha \vab{u}^{p-2} u \quad \text{in} \ \RN.
\end{equation}

\begin{proposition}\label{nonex1}
	\begin{enumerate}[label={\rm (\roman*)}]
		\item 
		Let $1 \leq N\le 4$, $q > 2$, $p > 2$, $\alpha>0$ and $u\in X_\rad$ be a solution to \eqref{eq_zero_mass}. Then $u\equiv 0$. 
		
		\item 
		Assume either $N=1,2$, $2<p,q<\infty$ or $N \geq 3$, $q>2$ and $2<p \leq 2^*$. 
		Let $u \in X$ be a solution to \eqref{eq_zero_mass}. Then $u \equiv 0$. 
	\end{enumerate}
\end{proposition}

From \cref{nonex1}, \eqref{A} holds under either $1 \leq N \leq 4$ or $N \geq 5$ and $2<p \leq 2^*$. 
The remaining case is when $N \geq 5$ and $p>2^*$. 
Our final result states that in general \eqref{A} does not hold in this case. 
To state the result, define $\wh{X}$ by 
\[
\begin{aligned}
	&\wh{X}\coloneq \Set{ u \in \cD^{1,2} (\RN) | \vab{\nabla u} \in L^q(\RN) }, \quad 
	\wh{X}_\rad \coloneq \Set{ u \in \wh{X} | u(x) = u(\vab{x}) }, 
	\\
	&\Vab{u}_{\wh{X}} \coloneq \Vab{\nabla u}_2 + \Vab{\nabla u}_q;
\end{aligned}
\]
see \cref{MountZero} for details. 

\begin{theorem}\label{t:zeromass}
	Let $N \geq 3$, $q > 2$ and $p \in (2^*,q^*)$. 
	Then \eqref{eq_zero_mass} admits infinitely many solutions $(u_k)_{k \in \N} \subset \wh{X}_\rad$ with $\Vab{u_k}_{\wh{X}} \to \infty$. 
	Furthermore, when $N \geq 5$, $u_k \in X_\rad$ holds for each $k \geq 1$. 
	Consequently, when $N \geq 5$ and $p \in (2^*,p_q)$, \eqref{A} does not hold provided $\inf_{k \geq 1} \Vab{u_k}_2 < m$. 
\end{theorem}

Regarding \cref{nonex1}, we exploit the arguments in \cite{BaSo17, IkTa19}, considering separately several cases according to the space dimension $N$, 
and making use of a comparison principle together with the Pohozaev identity in \cref{App}. 
The first part of \cref{t:zeromass}, on the other hand, follows from \cite[Theorem 1.7]{ByIkMaMa24}, 
while the fact that such a solution belongs to $X_\rad$ is obtained through a decay estimate $\vab{u(x)} \leq C(1+\vab{x})^{-N+2}$, which is interesting in its own right. 
The proof of this estimate is rather technical, due to the presence of multiple parameters and their mutual relationships. 
Since we can consider \eqref{eq_zero_mass} in $\wh{X}$ when $p=2^*$ (the corresponding functional is of class $C^1$), 
it is worth mentioning that the above decay estimate cannot be obtained for $p = 2^*$ when $N \geq 5$; see \cref{Rem:sharpness}. 
Thus, $p > 2^*$ in \cref{t:zeromass} is essential for the above decay estimate.

\medskip

The present paper is organized as follows. In \cref{s:pre}, we recall some classical results, such as the Gagliardo–Nirenberg inequality, 
and present a preliminary lemma concerning the energy functional $E_\alpha$ defined in \eqref{defE}. 
\cref{s:Ex-gl-mi} deals with the global minimization problem, where $\alpha$ varies with respect to $\alpha_0(m)$, 
and concludes with the proof of \cref{theomin}. \cref{Secloc} focuses on a local minimization problem in a suitable left neighborhood of $\alpha_0(m)$, 
and ends with the proof of \cref{theorem:local}. A mountain-pass type solution, established in \cref{t:ex-mp}, is derived in \cref{Secmou} under condition \eqref{A}. 
Sufficient conditions for the validity and failure of \eqref{A} are provided in \cref{SecLio,MountZero}, respectively, 
where Liouville-type nonexistence results (\cref{nonex1}) and existence results (\cref{t:zeromass}) are obtained in the zero-mass case $\lambda=0$. 
Finally, \cref{App} is devoted to regularity issues and the Pohozaev identity.



\subsubsection*{Notation}

Given any $r\in(1,\infty]$, we denote by $r' \coloneq\frac{r}{r-1}$ the H\"older conjugate of $r$ (with the position $r'=1$ if $r=\infty$). 
Moreover, $p^* \coloneq \frac{Np}{(N-p)_+}$ indicates the Sobolev conjugate of $p$. 

We indicate with $B_r(x)$ the $\R^N$-ball of center $x\in\R^N$ and radius $r>0$, omitting $x$ when it is the origin. For any $N$-dimensional Lebesgue measurable set $\Omega$, by $|\Omega|$ we mean its $N$-dimensional Lebesgue measure. 

If a sequence $(u_n)_{n \in \N}$ strongly converges to $u$, we write $u_n\to u$; if the convergence is in weak sense, we use $u_n\rightharpoonup u$. 

Given any measurable set $\Omega\subseteq \R^N$ and $q\in[1,\infty]$, $L^r(\Omega)$ stands for the standard Lebesgue space, whose norm will be indicated with $\Vab{\cdot}_{L^r(\Omega)}$, or simply $\Vab{\cdot}_r$ when $\Omega=\R^N$, 
while $H^1(\R^N)$ is the standard Sobolev space endowed with the norm $\Vab{\cdot}_2+\Vab{\nabla\cdot}_2$ and $H^1_\rad(\R^N)$ is its radial counterpart.

When $N \geq 3$, we will also use the Beppo Levi space $\cD^{1,2}(\R^N)$, 
which is the closure of the set $C^\infty_c(\R^N)$ of the compactly supported test functions with respect to the norm $\Vab{u}_{\cD^{1,2}(\R^N)} \coloneq \Vab{\nabla u}_2$.
Sobolev's inequality ensures the embedding $\cD^{1,2}(\R^N) \hookrightarrow L^{p^*}(\R^N)$ and $\cD^{1,2} (\R^N)$ can be expressed as 
$\cD^{1,2}(\R^N) = \set{ u\in L^{p^*}(\R^N) | \vab{ \nabla u }\in L^p(\R^N)}$.

In the sequel, the letter $C$ denotes a positive constant which may change its value at each passage; subscripts on $C$ emphasize its dependence on the specified parameters.

\section{Preliminaries}
\label{s:pre}

Except for the zero mass problem, we work on the following function space: 
\begin{equation*}
	\begin{aligned}
		X &\coloneq \Set{ u \in H^1(\RN) | \,\,|\nabla u| \in L^q(\RN) }, \quad 
		\Vab{u}_X \coloneq \Vab{u}_{H^1(\RN)} + \Vab{\nabla u}_q.
	\end{aligned}
\end{equation*}
Remark that
\begin{equation}\label{embeddings}
X \subset L^r(\RN) \quad \text{for any $r$} \in \begin{dcases}
	[2,q^*] & \text{if} \ q<N,\\
	[2,\infty) & \text{if} \ q = N, \\
	[2,\infty] & \text{if} \  q > N.
\end{dcases}
\end{equation}
In particular, $ X \subset L^q(\RN)$ holds.

Now we recall the Gagliardo–Nirenberg inequality, see \cite[p.125]{N59}.

\begin{lemma}
Let $p\in(2,r^*)$, $r>\frac{2N}{N+2}$. Then there exists a positive constant $C_r>0$ such that
\begin{equation}\label{GNq}
\Vab{u}_{p} \leq C_r \Vab{\nabla u}_{r}^{\nu_{p,r}} \Vab{u}_{2}^{(1-\nu_{p,r})}, \quad  u \in {L}^{2}(\R^{N}), \quad  |\nabla u| \in {L}^{r}(\R^{N}), 
\end{equation}
where $\nu_{p,r} \in (0,1)$ is given by
\[\nu_{p,r} \coloneq \frac{Nr}{r(N+2) - 2N} \cdot \frac{p-2}{p}. \]
\end{lemma}

To state the next lemma, we recall $K$ in \eqref{def-Smrho} and prepare 
\begin{equation*}
	Q_\alpha(u) \coloneq
	\Vab{\nabla u}_2^2 + \frac{1}{q}\ab( \frac{N+2}{2} q - N ) \Vab{\nabla u}_q^q - \frac{\alpha N}{p} \pab{ \frac{p}{2} - 1  } \Vab{u}_p^p.
	\label{def-Q}
\end{equation*}

\begin{lemma}\label{lemma:I}
	Let $N \geq 1$, $q>2$, $p_2 < p < p_q $ and $\alpha\ge 0$. Then the following statements hold:
	\begin{enumerate}[label={\rm (\roman*)}]
		\item 
		for any bounded sequence $(u_n)_{n \in \N}$ in $X$,
		if $\Vab{\nabla u_n}_2 \to 0$ or $\Vab{\nabla u_n}_q \to 0$ or $\Vab{u_n}_{\wt{p}} \to 0$ for some $\wt{p} \in [2,q^*)$ holds, then 
		$\Vab{u_n}_p \to 0$; 
		
		\item
		there exists $C = (p,q, N, m, \alpha)> 0$  such that
		\begin{equation*}
			E_\alpha (u) \geq \frac{1}{2} \Vab{\nabla u}_2^2 +  \frac{1}{q}\|\nabla u\|^q_{q} - C\|\nabla u\|^{\nu_{p,q}p}_{q}
		\end{equation*}
		for any $u \in X$ satisfying $\|u\|^2_{2} \leq m$; since $p \nu_{p,q} < q$, $E_\alpha$ is coercive on $S_m$;
		
		\item 
		for any $\wh{\alpha}>0$ there exists $\rho_0 = \rho_0( m, \wh{\alpha}) > 0$ small enough such that, 
		for every $\alpha \in (0,\wh{\alpha}]$ and  $u \in X$ 
		satisfying both $\Vab{u}^2_{2} \leq m$ and $K(u) \leq \rho_0(m,\wh{\alpha})$, we have respectively
		\begin{equation}\label{Eest}
			E_\alpha(u) \geq \frac{1}{2q} K(u)
		\end{equation}
		and 
		\begin{equation}\label{Pest}
			Q_\alpha(u) \geq \frac{1}{2} K(u). 
		\end{equation}
	\end{enumerate}
\end{lemma}

\begin{proof}
	(i) When $\Vab{\nabla u_n}_2 \to 0$ (resp. $\Vab{\nabla u_n}_q \to 0$), 
	\eqref{GNq} with $r=2$ (resp. $r=q$) gives 
	$\Vab{u_n}_{s} \to 0$ for any $s \in (2,2^*)$ (resp. $s \in (2,q^*)$). 
	Since $(u_n)_{n \in \N}$ is bounded in $X$, \eqref{embeddings}, $p \in (2,q^*)$ and 
	the interpolation inequality yield $\Vab{u_n}_p \to 0$. 
	Similarly, when $\Vab{u_n}_{\wt{p}} \to 0$ holds, $\Vab{u_n}_p \to 0$ can be proved via the interpolation inequality 
	and the boundedness of $(u_n)$ in $X$.

	(ii) Applying \eqref{GNq} with $r=q$, we get
	\[
	E_\alpha(u) \ge \frac{1}{2} \Vab{\nabla u}_2^2 + \frac{1}{q} \Vab{\nabla u}^q_q - \frac{\alpha}{p}C_q^p \Vab{\nabla u}_{q}^{\nu_{p,q}p} m^{(1-\nu_{p,q})p/2},
	\]
	and the coerciveness of $E$ follows since $q>\nu_{p,q}p$ by \eqref{middle}.
	
	(iii) 
	Let $\wh{\alpha} > 0$ be given and suppose $\alpha \in (0,\wh{\alpha}]$. 
	When $p_2 < p \leq 2^*$, applying \eqref{GNq} with $r=2$ (or Sobolev's inequality) implies that for each $u \in X$ with $\Vab{u}_2^2 \leq m$,
	\[
	\begin{aligned}
		E_\alpha (u) 
		&\geq 
		\frac{1}{2} \Vab{\nabla u}_2^2 + \frac{1}{q} \Vab{\nabla u}^q_q  - \frac{ \wh{\alpha} }{p}C_2^p \Vab{\nabla u}_{2}^{\nu_{p,2}p} 	m^{(1-\nu_{p,2})p/2}
		\\
		&\geq 
		\frac{1}{q} K(u) - \frac{ \wh{\alpha} }{p}C_2^p K(u)^{\nu_{p,2}p / 2} 	m^{(1-\nu_{p,2})p/2}.
	\end{aligned}
	\]
	Since $\nu_{p,2} p > 2$ holds by virtue of \eqref{middle}, 
	for a sufficiently small $\rho_0 = \rho_0(m, \wh{\alpha} ) > 0$, if $u \in X$ satisfies $\Vab{u}_2^2 \leq m$ and $K(u) \leq \rho_0(m,\wh{\alpha})$, 
	then \eqref{Eest} holds. 
	For \eqref{Pest}, by $q < \frac{N+2}{2} q - N$, a similar argument works and \eqref{Pest} may be proved. 
	
	When $2^*<p < p_q$, choose $r \in (2, \min \set{N,q} )$ so that $(p_q=)\max \set{q, p_q} < r^*$.
	By dividing the cases $\vab{s} \leq 1$ and $\vab{s} >1$ separately, one can prove
	\[
	\vab{s}^p \leq \vab{s}^{2^*} + \vab{s}^{r^*} \quad \text{for all $s \in \R$},
	\]
	giving
	\[
	\Vab{u}_p^p \leq \Vab{u}_{2^*}^{2^*} + \Vab{u}_{r^*}^{r^*} \leq C \Vab{\nabla u}_2^{2^*} + C \Vab{\nabla u}_r^{r^*}. 
	\]
	Since $2 < r < q$, there exists $\theta \in (0,1)$ such that 
	\[
	\frac{1}{r} = \frac{\theta}{2} + \frac{1-\theta}{q}. 
	\]
	Then the interpolation inequality followed by Young's inequality gives 
	\[
	\Vab{\nabla u}_r \leq \Vab{\nabla u}_2^\theta \Vab{\nabla u}_q^{1-\theta} \leq \theta \Vab{\nabla u}_2 + (1-\theta) \Vab{\nabla u}_q.
	\]
	In particular, 
	\[
	\Vab{\nabla u}_r^{r^*} \leq C \Vab{\nabla u}_2^{r^*} + C \Vab{\nabla u}_q^{r^*}. 
	\]
	Therefore, 
	\[
	\Vab{u}_p^p \leq C \ab(  \Vab{\nabla u}_2^{2^*} + \Vab{\nabla u}_2^{r^*} ) + C \Vab{\nabla u}_q^{r^*}
	\]
	and this implies that if $u \in X$ satisfies $\Vab{u}_2^2 \leq m$ and $K(u) \leq \rho$, then 
	\begin{equation*}
		\begin{aligned}
			E_\alpha(u) 
			&\geq 
			\frac{1}{2} \Bab{ 1 - C \wh{\alpha} \ab( \Vab{\nabla u}_2^{2^*-2} + \Vab{\nabla u}_2^{ r^* -2 } ) } \Vab{\nabla u}_2^2 
			+ \frac{1}{q} \ab( 1 - C \wh{\alpha} \Vab{\nabla u}_q^{r^*-q} ) \Vab{\nabla u}_q^q
			\\
			& \geq 
			\frac{1}{2} \Bab{ 1 - C \wh{\alpha} \ab(  \rho^{(2^*-2)/2} + \rho^{ (r^*-2)/2 } )  } \Vab{\nabla u}_2^2 
			+ \frac{1}{q} \Bab{ 1 - C \wh{\alpha} \rho^{ (r^*-q)/q } } \Vab{\nabla u}_q^q.
		\end{aligned}
	\end{equation*}
	Since $2 < q < r^*$, if $\rho > 0$ is sufficiently small, then \eqref{Eest} holds. 
	Again, the same argument works for \eqref{Pest} and we omit the details. 
\end{proof}

\section{Minimizers with negative energy}
\label{s:Ex-gl-mi}

In the present section, we will consider the following minimizing problem
\[
e_\alpha(m) \coloneq \inf_{u \in S_m} E_\alpha(u),
\]
where $E_\alpha$ and $S_m$ are respectively defined in \eqref{defE} and \eqref{defSm}. 
We first prove the following result:

\begin{theorem}\label{prop:min}
Let $N \geq 1 $, $q>2$, $p_2< p <p_q $ and $m>0$. Then $-\infty < e_\alpha(m) \leq 0$ holds for every $\alpha > 0$. 
Furthermore, there exists $\alpha_0(m)>0$ such that the following assertions hold: 
\begin{enumerate}[label={\rm (\roman*)}]
	\item 
	$e_\alpha(m) < 0$ if and only if $\alpha > \alpha_0 (m)$; 
	
	\item 
	when $\alpha = \alpha_0(m)$, there is a $u_0 \in S_m$ such that $u_0$ is radially symmetric and 
	satisfies $E_{\alpha_0(m)} (u_0) = 0 = e_{\alpha_0(m)} (m)$, 
	that is, $u_0$ is a radial minimizer;
	
	\item 
	when $0<\alpha < \alpha_0(m)$, $e_{\alpha} (m)$ is never attained. 	
\end{enumerate}
\end{theorem}

\begin{proof}
The assertion $-\infty < e_\alpha(m)$ follows from \cref{lemma:I} (ii). 
On the other hand, for $u \in S_m$ and $\theta > 0$, define
\[
u_\theta (x) \coloneq \theta^{N/2} u(\theta x). 
\]
Notice that 
\begin{equation}\label{sc}
\begin{aligned}
	&\Vab{u_\theta}_2^2 = \Vab{u}_2^2, 
	& & 
	\Vab{u_\theta}_p^p = \theta^{Np/2 - N} \Vab{ u }_p^p = \theta^{p \delta_p} \Vab{u}_p^p, 
	\\
	&\Vab{\nabla u_\theta}_2^2 = \theta^2 \Vab{\nabla u}_2^2, 
	& &
	\Vab{\nabla u_\theta}_q^q = \theta^{ q\delta_q + q } \Vab{\nabla u}_q^q,
\end{aligned}
\end{equation}
where
$\delta_s \coloneq \frac{N}{2s} (s-2)$, which is equivalent to $s \delta_s = \frac{N}{2} (s-2)$. 
In particular, $u_\theta \in S_m$ for every $\theta > 0$ provided $u \in S_m$. Since 
\[
E_\alpha(u_\theta) = \theta^2 
\ab[ \frac{1}{2} \Vab{\nabla u}_2^2 + \frac{  \theta^{q\delta_q + q -2}  }{q} \Vab{\nabla u}_q^q - \frac{\alpha}{p} \theta^{p\delta_p -2} \Vab{u}_p^p ].
\]
it is easily seen from $p\delta_p > 2$ that $E_\alpha(u_\theta) \to 0$ as $\theta \to 0^+$ and $e_\alpha(m) \leq 0$ holds.

We next prove the existence of $\alpha_0(m)$. For $u \in S_m$, set 
\[
\psi_{\alpha,u}(\theta) \coloneq \frac{1}{2} \Vab{\nabla u}_2^2 + \frac{  \theta^{q\delta_q + q -2}  }{q} \Vab{\nabla u}_q^q - \frac{\alpha}{p} \theta^{p\delta_p -2} \Vab{u}_p^p
= \theta^{-2} E_\alpha(u_\theta).
\]
From $q\delta_q + q > p\delta_p > 2$ and 
\[
\psi_{\alpha,u}'(\theta) = \theta^{ p\delta_p - 3 } 
\ab[ \frac{q(1+\delta_q)-2}{q} \theta^{ q(1+\delta_q) - p \delta_p } \Vab{\nabla u}_q^q 
- \frac{\alpha}{p} \ab( p \delta_p - 2 ) \Vab{u}_p^p ],
\]
it follows that $\psi_{\alpha,u}$ takes the global minimum at $\theta = \bar{\theta}_{\alpha,u}$ where 
\begin{equation}\label{def:bartheta}
	\bar{\theta}_{\alpha,u} \coloneq 
	\ab[ \frac{ q\alpha (p\delta_p-2)  \Vab{u}_p^p }{ p \Bab{ q (1+\delta_q) -2  } \Vab{\nabla u}_q^q}  ]^{ \frac{ 1 }{ q(1+\delta_q) - p\delta_p } }
\end{equation}
and the global minimum of $\psi_{\alpha,u}$ is given by 
\begin{equation}\label{e:psi-gm}
	\begin{aligned}
		&\psi_{\alpha,u} \ab( \bar{\theta}_{\alpha,u} ) 
		\\
		= \ &
		\frac{1}{2} \Vab{\nabla u}_2^2 
		+ \frac{1}{q} \ab[ \frac{q\alpha (p\delta_p - 2) \Vab{u}_p^p }{ p \Bab{ q (1+\delta_q) -2 } \Vab{\nabla u}_q^q } ]^{ \frac{ q(1+\delta_q) - 2 }{q(1+\delta_q) - p\delta_p} } 
		\Vab{\nabla u}_q^q
		\\
		&
		- \frac{\alpha}{p} 
		\ab[ \frac{q\alpha (p\delta_p - 2) \Vab{u}_p^p }{ p \Bab{ q (1+\delta_q) -2 } \Vab{\nabla u}_q^q } ]^{ \frac{ p\delta_p - 2 }{q(1+\delta_q) - p\delta_p} } \Vab{u}_p^p
		\\
		= \ &  
		\frac{1}{2} \Vab{\nabla u}_2^2  
		- \alpha^{ \frac{ q(1+\delta_q) - 2 }{ q ( 1+\delta_q) - p\delta_p } } 
		\ab[ \frac{q  (p\delta_p - 2) }{ p \Bab{ q (1+\delta_q) -2 } } ]^{ \frac{ p\delta_p - 2 }{q(1+\delta_q) - p\delta_p} } \\
		&\hspace{5cm} \cdot
		\ab[ \frac{q(1+\delta_q) - p\delta_p}{ p \Bab{ q(1+\delta_q) - 2 }  } ] 
		\Vab{u}_p^{ p \cdot \frac{ q(1+\delta_q) -2 }{q(1+\delta_q) - p\delta_p}  } 
		\Vab{\nabla u}_q^{ q \cdot \frac{2-p\delta_p}{q(1+\delta_q) - p\delta_p} }.
	\end{aligned}
\end{equation}
From $2 < p\delta_p < q (1+\delta_q)$, it follows that $\psi_{\alpha,u}( \bar{\theta}_{\alpha,u} ) < 0$ is equivalent to 
\[
\alpha^{ \frac{ q(1+\delta_q) - 2 }{ q ( 1+\delta_q) - p\delta_p } } 
> 
\ab[ \frac{q  (p\delta_p - 2) }{ p \Bab{ q (1+\delta_q) -2 } } ]^{ - \frac{ p\delta_p - 2 }{q(1+\delta_q) - p\delta_p} } 
\cdot 
\frac{ p \Bab{ q(1+\delta_q) - 2 } }{q(1+\delta_q) - p\delta_p}
\cdot 
\frac{1}{2} \cdot 
\frac{\Vab{\nabla u}_2^2 \Vab{\nabla u}_q^{ q \cdot \frac{p\delta_p-2}{q(1+\delta_q) - p\delta_p}  } }
{\Vab{u}^{ p \cdot \frac{ q(1+\delta_q) -2 }{q(1+\delta_q) - p\delta_p} } }
.
\]
To proceed further, we set 
\[
J(u) \coloneq 
\frac{\Vab{\nabla u}_2^2 \Vab{\nabla u}_q^{ q \cdot \frac{p\delta_p - 2 }{q(1+\delta_q) - p\delta_p} }}
{ \Vab{u}_p^{ p  \cdot \frac{ q(1+\delta_q) -2 }{q(1+\delta_q) - p\delta_p} } }
\quad \text{for $u \in X \setminus \set{0}$}
\]
and consider the following minimizing problem: 
\begin{equation}\label{mp}\begin{aligned}
d(m) &\coloneq \inf \Set{ J(u) | u  \in S_m}.
\end{aligned}\end{equation}
By defining 
\begin{equation}\label{alpha0}
	\alpha_0(m) 
	\coloneq  
	\Bab{
		\ab[ \frac{q  (p\delta_p - 2) }{ p \Bab{ q (1+\delta_q) -2 } } ]^{ - \frac{ p\delta_p - 2 }{q(1+\delta_q) - p\delta_p} } 
		\cdot 
		\frac{ p \Bab{ q(1+\delta_q) - 2 } }{q(1+\delta_q) - p\delta_p}
		\cdot 
		\frac{1}{2} \cdot d(m) }^{ \frac{ q(1+\delta_q) - p\delta_p }{q(1+\delta_q)-2} },
\end{equation}
the above computations show that $\alpha > \alpha_0(m)$ if and only if $e_\alpha(m) <0$. 
In addition, when $\alpha = \alpha_0(m)$, if there is a radial $u_0 \in S_m$ such that $J(u_0) = d(m)$, 
then $\psi_{\alpha_0(m),u_0} (\bar{\theta}_{\alpha_0(m),u_0}  ) = 0$ and $E( (u_0)_{ \ov{\theta}_{\alpha_0(m),u_0}} ) = 0$. 
Therefore, assertion (ii) holds provided $d(m)$ has a radial minimizer. 
Furthermore, when $0<\alpha < \alpha_0(m)$, it follows from \eqref{e:psi-gm}, \eqref{mp} and \eqref{alpha0} that for each $u \in S_m$, 
\[
\alpha^{ \frac{ q(1+\delta_q) - 2 }{ q ( 1+\delta_q) - p\delta_p } } 
< 
\ab[ \frac{q  (p\delta_p - 2) }{ p \Bab{ q (1+\delta_q) -2 } } ]^{ - \frac{ p\delta_p - 2 }{q(1+\delta_q) - p\delta_p} } 
\cdot 
\frac{ p \Bab{ q(1+\delta_q) - 2 } }{q(1+\delta_q) - p\delta_p}
\cdot 
\frac{1}{2}\cdot J(u),
\]
which implies 
\[
\begin{aligned}
	E_\alpha(u) 
	= \psi_{\alpha,u}(1) \geq 
	\psi_{\alpha,u}( \bar{\theta}_{\alpha,u} ) 
	> \frac{1}{2} \Vab{\nabla u}_2^2 - \frac{1}{2} \Vab{\nabla u}_2^2 = 0
	\quad \text{for each $u \in S_m$.}
\end{aligned}
\]
Hence, $e_\alpha(m) = 0$ and $E_\alpha(v) > 0$ for each $v \in S_m$, which means that 
there is no minimizer for $e_\alpha(m)$ and assertion (iii) holds.

To complete the proof, what remains to prove is $\alpha_0 (m) > 0$, namely $d(m) > 0$ and $d(m) $ is attained 
by a radial function. 
For this purpose, we consider $J(\tau u)$ for $\tau > 0$. By $q\delta_q = N (q-2)/2$ and $p\delta_p = N(p-2)/2$, we see 
\[
\begin{aligned}
	2 + \frac{q(p\delta_p-2)}{q(1+\delta_q) -p\delta_p} - \frac{ p \Bab{ q(1+\delta_q) -2 } }{q(1+\delta_q) - p\delta_p}
	&=
	\frac{ 2q(1+\delta_q) -2p\delta_p + pq \delta_p - 2q - pq - pq\delta_q + 2p }
	{q(1+\delta_q) - p\delta_p}
	\\
	&=
	\frac{ q\delta_q ( 2 -p) + (q-2) p \delta_p + p(2-q) }
	{q(1+\delta_q) - p\delta_p}
	\\
	&=
	\frac{ \frac{N}{2} (q-2)  ( 2 -p) + (q-2) \frac{N}{2} (p-2) + p(2-q) }
	{q(1+\delta_q) - p\delta_p} 
	\\
	&= -
	\frac{p(q-2)}{q(1+\delta_q) - p\delta_p} < 0.
\end{aligned}
\]
Therefore, 
\begin{equation}\label{J:multi}
	\begin{aligned}
		J( \tau u ) 
		&= 
		\tau^{ - \frac{p(q-2)}{q(1+\delta_q) - p\delta_p} } J(u) 
		\quad \text{for every $\tau \in (0,\infty)$ and $u \in X \setminus \set{0}$}. 
	\end{aligned}
\end{equation}
In particular, $u \in S_1$ attains $d(1)$ if and only if $\sqrt{m} u$ does $d(m)$ and 
\begin{equation}\label{e:dmd1}
	d(m) = m^{ - \frac{p(q-2)}{2 \Bab{ q(1+\delta_q) - p\delta_p} } } d(1).
\end{equation}
Thus, it suffices to prove that $d(1) > 0$ and there is a radial minimizer for $d(1)$.

We first show 
\begin{equation}\label{rewritemp}
	\begin{aligned}
		d(1) &= \inf \Set{ J(u) |  u \in S_1, \Vab{u}_p = 1}
		\\
		& = \inf \Set{ J(u) | u \in S_1, \Vab{u}_p = 1, \text{$u(x) = u(\vab{x}) \geq 0$: nonincreasing in $\vab{x}$}}.
	\end{aligned}
\end{equation}
Indeed, for the first equality, let $v \in S_m$ and consider $v_\theta$ where $\theta = \Vab{u}_p^{-1/\delta_p}$. From \eqref{sc} and 
\[
2 + q(\delta_q + 1) \frac{p\delta_p - 2}{q(1+\delta_q) - p\delta_p} - p\delta_p \frac{q(1+\delta_q) - 2}{q(1+\delta_q) - p \delta_p} 
= 2 + \frac{-2q(\delta_q+1) + 2 p \delta_p }{q(1+\delta_p) - p\delta_p} = 0,
\]
it follows that $J(v_\theta) = J(v)$. 
Hence, the first equality in \eqref{rewritemp} holds. For the second equality in \eqref{rewritemp}, 
let $u \in C^\infty_c(\RN)$ fulfill $\Vab{u}_2^2 = 1$ and $\Vab{u}_p = 1$, and 
$u^*$ denote the symmetric decreasing rearrangement of $u$. 
By \cite[Lemma 1]{Ta76}, we have 
\[
\Vab{u}_r=\Vab{u^*}_r=1 \quad (r=2,p), \quad \Vab{\nabla u^*}_r \leq \Vab{\nabla u}_r \quad (r=2,q), 
\quad J(u^*) \leq J(u).
\]
The density of $C^\infty_c(\RN)$ with $\Vab{u}_2 = 1$ yields the second equality in \eqref{rewritemp}.


\smallskip 

\noindent
\textbf{Claim 1:}
\textsl{$d(1) > 0$.}

\smallskip

Let $(u_n)_{n \in \N} \subset S_1$ satisfy $\Vab{u_n}_p = 1$ and $J(u_n) \to d(1)$. 
By the Gagliardo--Nirenberg inequality \eqref{GNq}, 
\[
1 = \Vab{u_n}_p \leq C_q \Vab{\nabla u_n}_q^{\nu_{p,q}}. 
\]
In particular, 
\begin{equation}\label{ineq:inf-graunlq}
	0< \inf_{n \in \N} \Vab{\nabla u_n}_q \eqcolon c_0. 
\end{equation}
Hence, 
\begin{equation}\label{bdd-L2-gra-un}
	d(1) + o(1) = J(u_n) \geq c_0^{ q \cdot \frac{p\delta_p -2 }{ q(1+\delta_q) - p\delta_p } } \Vab{\nabla u_n}_2^2.
\end{equation}
Thus, $(\nabla u_n)_{n \in \N}$ is bounded in $L^2(\R^N)$. 
From now on, the arguments are divided into two cases: (I) $p_2 < p \leq 2^*$, (II) $2^* < p < p_q$.

In case (I), if $d(1) = 0$ holds, then \eqref{bdd-L2-gra-un} yields $\Vab{\nabla u_n}_2 \to 0$. 
Thus, \eqref{GNq} with $r=2$ (or Sobolev's inequality for $p=2^*$) gives $\Vab{u_n}_p \to 0$, which is a contradiction since $\Vab{u_n}_p = 1$. 
Therefore, $d(1) > 0$ holds in case (I).

In case (II), from $2< q$ and $2^* < p < p_q$, 
there exists $r \in (2, \min \set{N,q} )$ such that 
\[
r^* = p, \quad  \text{that is}, \quad \frac{1}{r} = \frac{1}{N} + \frac{1}{p} = \frac{p+N}{Np}. 
\]
Then for each $u \in S_1$ with $\Vab{u}_p = 1$, Sobolev's inequality and the interpolation inequality lead to 
\begin{equation*}\label{gra-u_n2toq}
	1 = \Vab{u}_p \leq C \Vab{\nabla u}_r \leq C \Vab{\nabla u}_2^\theta \Vab{\nabla u}_q^{1-\theta},
\end{equation*}
where 
\[
\frac{1}{r} = \frac{\theta}{2} + \frac{1-\theta}{q} \ \iff \ \theta = \frac{pq+Nq-pN}{Np} \cdot \frac{2}{q-2}. 
\]
Therefore, there exists $C'>0$ such that 
\[
C' \Vab{\nabla u}_q^{ 2 - 2/\theta } \leq \Vab{\nabla u}_2^2 \quad \text{for every $u \in S_1$ with $\Vab{u}_p = 1$}.
\]
Substituting this inequality in $J(u)$ gives 
\begin{equation}\label{bdd-J}
	J(u) = \Vab{\nabla u}_2^2 \Vab{\nabla u}_q^{ q \cdot \frac{p\delta_p - 2}{q(1+\delta_q) - p\delta_p} } 
	\geq C' \Vab{\nabla u}_q^{ 2 - \frac{2}{\theta} + q \cdot \frac{p\delta_p - 2}{q(1+\delta_q) - p\delta_p} } 
	\quad \text{for every $u \in S_1$ with $\Vab{u}_p = 1$}. 
\end{equation}
Since $p_2 < p < p_q < q^*$, a direct computation gives 
\[
2 - \frac{2}{\theta} + q \cdot \frac{p\delta_p-2}{q(1+\delta_q) - p\delta_p} 
= \frac{Npq(p-2)(q-2)}{ (2q + Nq - Np) (Nq + pq - Np) } > 0.
\]
Applying \eqref{bdd-J} for $u_n$ and recalling \eqref{ineq:inf-graunlq} yield $d(1) > 0$.

\smallskip 

\noindent
\textbf{Claim 2:}
\textsl{there exists a radial $u_1 \in S_1$ such that $J(u_1) = d(1)$.}

\smallskip

Let $(u_n)_{n \in \N}$ satisfy $\Vab{u_n}_2 = 1 = \Vab{u_n}_p$, $J(u_n) \to d(1) > 0$ 
and assume that $u_n(x) = u_n(\vab{x}) \geq 0$ is nonincreasing in $\vab{x}$. 
Similarly to Claim 1, we divide the proof into two cases.

In case (I), from the proof of Claim 1, $(\Vab{\nabla u_n}_2)_{n \in \N}$ is bounded, and furthermore 
notice that $\inf_{n \in \N} \Vab{\nabla u_n}_2 > 0$; 
otherwise, $\Vab{\nabla u_n}_2 \to 0$ gives a contradiction $1=\Vab{u_n}_p \to 0$ as in the proof of Claim 1. 
Since $(\Vab{\nabla u_n}_2)_{n\in \N}$ is bounded and stays away from $0$, 
$(\nabla u_n)_{n \in \N}$ is bounded in $L^2(\R^N) \cap L^q(\RN)$ in view of $J(u_n) \to d(1) > 0$. 
Taking a subsequence if necessary, we may assume $u_n \rightharpoonup u_\infty$ weakly in $H^1_\rad(\R^N)$ and $L^p(\RN)$, and 
$\nabla u_n \rightharpoonup \nabla u_\infty$ weakly in $L^q(\RN)$. 
When $N \geq 2$, since the embedding $H^1_\rad(\RN) \subset L^s(\RN)$ is compact for $s \in (2,2^*)$ (\cite[Theorem A.I']{BeLi83}), 
the interpolation inequality implies that 
$X_\rad \subset L^s(\RN)$ is compact for any $s \in (2,q^*)$. 
In particular, $\Vab{u_n}_p \to \Vab{u_\infty}_p$ and hence $\Vab{u_\infty}_p = 1$ and $u_\infty \not \equiv 0$. 
We claim that $\Vab{u_\infty}_2 = 1$. In fact, if $0<\Vab{u_\infty}_2 < 1$, then 
by setting $\tau_\infty \coloneq 1/ \Vab{u_\infty}_2 \in (1,\infty)$, 
it follows from \eqref{J:multi} and the weak continuity of norms that 
\[
0 < d(1) \leq J(\tau_\infty u_\infty) = \tau_\infty^{ - \frac{p(q-2)}{ q(1+\delta_q) - p\delta_p } } J(u_\infty) < J(u_\infty) 
\leq \liminf_{n \to \infty} J(u_n) = d(1),
\]
which is a contradiction. Therefore, $\Vab{u_\infty}_2 = 1$. Similarly to the above, 
we obtain $d(1) \leq J(u_\infty) \leq \liminf_{n \to \infty} J(u_n) = d(1)$ and $d(1) = J(u_\infty)$. 
Thus, $u_\infty$ is a desired radial minimizer.

On the other hand, when $N=1$, we may also suppose that $u_n \to u_\infty$ in $C_{\rm loc} (\R)$. 
Thus, $u_\infty(x) = u_\infty(\vab{x}) \geq 0$ and $u_\infty$ is nonincreasing in $[0,\infty)$. 
Let us prove $\Vab{u_\infty}_p = 1$. For any given $\e>0$, 
since $u_\infty(x) \to 0$ as $\vab{x} \to \infty$ by $u_\infty \in H^1(\R)$, 
there exists $R_\e > 0$ such that 
$ 0 \leq u_\infty(x) < \e$ holds for each $x \in \R$ with $\vab{x} \geq R_\e$. 
Since $u_n(\pm R_\e) \to u_\infty(R_\e)$ as $n \to \infty$ and $u_n$ is monotone, 
there exists $n_\e \in \N$ such that 
$0 \leq u_n(x) \leq \e$ for each $n \in \N$ and $x \in \R$ with $n \geq n_\e$ and $\vab{x} \geq R_\e$. 
This together with $u_\infty (x) \to 0$ as $\vab{x} \to \infty$ and $u_n \to u_\infty$ in $C_{\rm loc} (\R)$ leads to $\Vab{u_n - u_\infty}_\infty \to 0$. 
The interpolation inequality gives $\Vab{u_n - u_\infty}_p \to 0$. 
Thus, $\Vab{u_\infty}_p = 1$ holds. 
Then we may argue in a similar way to the case $N \geq 2$ and get a radial minimizer corresponding to $d(1)$.

In case (II), due to \eqref{bdd-L2-gra-un} and \eqref{bdd-J} with $u=u_n$, 
$(\nabla u_n)_{n \in \N}$ is bounded in $L^2(\RN) \cap L^q(\R^N)$ thanks to $J(u_n) \to d(1) > 0$. 
Then the rest of the argument is the same as case (I) and 
we can find a radial minimizer for $d(1)$.

From Claims 1 and 2, we conclude assertions (i) and (ii), and this completes the proof. 
\end{proof}

\begin{remark}\label{Rem:alp0m}
	From \eqref{alpha0} and \eqref{e:dmd1}, the dependence of $\alpha_0(m)$ on $m$ is expressed as follows: 
	\[
		\alpha_0(m) 
		=
	\Bab{
		\ab[ \frac{q  (p\delta_p - 2) }{ p \Bab{ q (1+\delta_q) -2 } } ]^{ - \frac{ p\delta_p - 2 }{q(1+\delta_q) - p\delta_p} } 
		\cdot 
		\frac{ p \Bab{ q(1+\delta_q) - 2 } }{q(1+\delta_q) - p\delta_p}
		\cdot 
		\frac{1}{2} \cdot d(1) }^{ \frac{ q(1+\delta_q) - p\delta_p }{q(1+\delta_q)-2} } 
		m^{ - \frac{ p(q-2) }{2\Bab{ q(1+\delta_q) - 2 }} }.
	\]
	Moreover, by denoting 
	\[
	\cG_J \coloneq \Set{ u \in S_1 | \Vab{u}_p = 1, \ J(u) = d(1)  },
	\]
	 when $\alpha = \alpha_0(m)$, by \eqref{def:bartheta} and \eqref{e:psi-gm}, 
	 the set $\scrG_0$ of all minimizers for $e_{\alpha_0(m)} (m)$ is given by 
	 \[
	 \scrG_0 = \Set{ \ab( \sqrt{m} u)_{\theta_{\alpha_0(m), \sqrt{m} u}} | u \in \cG_J  },
	 \]
	 where 
	 \[
	 \theta_{\alpha_0(m), \sqrt{m} u} \coloneq 
	 \ab[ \frac{ q\alpha (p\delta_p-2)  m^{(p-q)/2} }{ p \Bab{ q (1+\delta_q) -2  } \Vab{\nabla u}_q^q}  ]^{ \frac{ 1 }{ q(1+\delta_q) - p\delta_p } }.
	 \]
\end{remark}

Before going into the proof of \cref{theomin}, we proceed with intermediate results.

\begin{proposition}
Let $m>0$ and $\alpha > 0$. Then the subadditivity of $e_\alpha(m)$ holds: 
\begin{equation}\label{eq:subadd}
	e_\alpha(m) \leq e_\alpha \ab( \tau )  + e_\alpha \ab( m - \tau ) \quad \text{for all $\tau \in (0,m)$}. 
\end{equation}
\end{proposition}

\begin{proof}
For any $\omega \in S^{N-1}$ and 
$\varphi , \psi \in C^\infty_c(\RN)$ with $\Vab{\varphi}_2^2 = \tau$ and $\Vab{\psi}_2^2 = m - \tau$, 
if $k$ is sufficiently large, then 
\[
\supp \varphi \cap \supp \psi ( \cdot - k\omega) = \emptyset
\]
and 
\[
e_\alpha (m) \leq E_\alpha \ab( \varphi + \psi ( \cdot - k\omega) ) = E_\alpha (\varphi) + E_\alpha(\psi).
\]
Recalling that $\varphi$ and $\psi$ are arbitrary, we get \eqref{eq:subadd} by the density of $C^\infty_c(\RN)$ in $X$. 
\end{proof}

Notice that by \eqref{eq:subadd} and $e_\alpha(\tau) \leq 0$ for any $\tau > 0$, then $e_\alpha(\cdot)$ is monotone, namely
\begin{equation}\label{mono}
	\text{$(0,\infty) \ni \tau \mapsto e_\alpha (\tau) \in \R$ is nonincreasing.}
\end{equation}

In the next proposition we list other properties of $e_\alpha(m)$ (cf. \cite[Lemma 3.5]{DiJiPu25} and \cite[Lemma 4.3 and Corollary 4.4]{YaQiZo22}).

\begin{proposition}\label{propem}
	Let $m_1 > 0$ be given. Then, 
	\begin{enumerate}[label={\rm (\roman*)}]
		\item 
		for each $\theta \geq 1$, $e_\alpha (\theta^N m_1) \leq \theta^N e_\alpha (m_1)$; 
		\item 
		if either $e_\alpha(m_1) < 0$ or $e_\alpha(m_1)$ is attained by $u_1 \in S_{m_1}$, 
		then for every $\theta > 1$, $e_\alpha (\theta^N  m_1) < \theta^N e_\alpha (m_1)$;
		\item 
		if $e_\alpha (m_1)$ is attained by $u_1 \in S_{m_1}$, then for every $m_2>0$, 
		$e_\alpha (m_1+m_2) < e_\alpha (m_1) + e_\alpha (m_2)$. 
	\end{enumerate}
\end{proposition}

\begin{proof} 
For (i) and (ii), we follow the ideas contained in \cite[Lemma 3.5]{DiJiPu25}, 
and for (iii) we argue as in \cite[Corollary 4.4]{YaQiZo22}.

(i) Choose $u \in S_{m_1}$ arbitrary and consider $u_\theta (x) \coloneq u(x/\theta)$ for $\theta \geq 1$.
It is immediate to see that 
\[
\Vab{u_\theta}_2^2 = \theta^N m_1, \quad \Vab{\nabla u_\theta}_q^q = \theta^{N-q} \Vab{\nabla u}_q^q, \quad 
\Vab{u_\theta}_p^p = \theta^{N} \Vab{u}_p^p.
\]
This fact and recalling $K$ in \eqref{def-Smrho} imply
\[
\begin{aligned}
	\odv*{E_\alpha(u_\theta)}{\theta} 
	&= \theta^{-1} \Bab{ \frac{N-2}{2} \Vab{\nabla u_\theta}_2^2 + \frac{N-q}{q} \Vab{\nabla u_\theta}_q^q 
	- \frac{\alpha N}{p} \Vab{u_\theta}_p^p }
	=
	N \theta^{-1} E_\alpha(u_\theta) - \theta^{-1} K(u_\theta) . 
\end{aligned}
\]
By 
\[
\theta^{N}\odv*{ \ab( \theta^{-N} E_\alpha(u_\theta) ) }{\theta} 
= - N \theta^{-1} E_\alpha(u_\theta) + \odv*{E_\alpha(u_\theta)}{\theta},
\]
it follows that 
\[
\theta^{N}\odv*{ \ab( \theta^{-N} E_\alpha(u_\theta) ) }{\theta} = - \theta^{-1} K(u_\theta).
\]
Thus, for any $\theta \geq 1$, 
\[
\theta^{-N} E_\alpha (u_\theta) - E_\alpha (u) = -  \int_1^\theta \tau^{-N-1} K(u_\tau )  \odif{\tau}.
\]
Noting $u_\theta \in S_{\theta^N m_1}$ and $K(u_\tau) \geq 0$ gives 
\begin{equation}\label{eq_diff_theta^Nm_1m_1}
e_\alpha (\theta^N m_1) \leq  E_\alpha (u_\theta) = \theta^N E_\alpha (u) - \theta^N \int_1^\theta \tau^{-N-1} K(u_\tau)  \odif{\tau} \leq \theta^N E_\alpha (u).
\end{equation}
Since $u \in S_{m_1}$ is arbitrary, $e_\alpha (\theta^N m_1) \leq \theta^N e_\alpha(m_1)$ holds.

(ii) 
Let $\theta > 1$. 
When $e_\alpha(m_1) < 0$, choose $(u_n)_{n \in \N} \subset S_{m_1}$ so that $E_\alpha (u_n) \to e_\alpha(m_1) < 0$. 
Then $(u_n)_{n \in \N}$ is bounded in $X$ thanks to \cref{lemma:I} (ii). 
Since $e_\alpha(m_1) < 0$, \cref{lemma:I} leads to $\inf_{n \geq 1} K(u_n) > 0$; 
otherwise $\liminf_{n \to \infty} E_\alpha(u_n) \geq 0$ holds. 
Notice that $\inf_{ n \geq 1, 1 \leq \tau \leq \theta } K( (u_n)_\tau ) \eqcolon \delta_0 > 0$. 
From \eqref{eq_diff_theta^Nm_1m_1} it follows that 
\[
\begin{aligned}
	e_\alpha \ab( \theta^N m_1 ) 
	& \leq \theta^N E_\alpha(u_n) - \theta^N \int_1^\theta \tau^{-N-1} K \ab( (u_n)_\tau) \odif{\tau}
	\\
	& \leq \theta^N E_\alpha(u_n) - \theta^N \delta_0 \int_1^\theta \tau^{-N-1} \odif{\tau} 
	\to \theta^N e_\alpha(m_1) - \frac{\theta^N \delta_0}{N} \ab( 1 - \theta^{-N} ).
\end{aligned}
\]
Hence, $e_\alpha (\theta^N m_1) < \theta^N e_\alpha(m_1)$ holds.

On the other hand, when $e_\alpha(m_1)$ is attained by $u_1 \in S_{m_1}$, 
we may use the above argument with $u_n = u_1$ and obtain 
\[
e_\alpha(\theta^N m_1) \leq \theta^N E_\alpha(u_1) - \theta^N \int_1^\theta \tau^{-N-1} K\ab( (u_1)_\tau ) \odif{\tau} < \theta^N e_\alpha (m_1). 
\]
This proves (ii).

(iii) The case $m_1 = m_2$ can be reduced into assertion (ii) with $\theta^N = 2$. 
When $0 < m_2 < m_1$, by rewriting
\[
m_1+m_2 = \frac{m_1+m_2}{m_1} \cdot m_1 \eqcolon \theta^N_1 m_1, \quad 
m_1 = \frac{m_1}{m_2} \cdot m_2 \eqcolon \theta_2^N m_2
\]
and by noting $\theta_1^N-1 = m_2/m_1 = \theta_2^{-N}$, assertions (ii) and (i) imply 
\[
e_\alpha (m_1+m_2) < \theta_1^N e_\alpha (m_1) = e_\alpha (m_1) + \theta_2^{-N} e_\alpha ( \theta_2^N m_2 ) 
\leq e_\alpha (m_1) + e_\alpha (m_2).
\]
On the other hand, if $m_1 < m_2$, then since $e_\alpha(m_1)$ is attained by $u_1 \in S_{m_1}$ and $m_1 < m_2$, 
assertion (ii) leads to $e_\alpha(m_2) < e_\alpha(m_1) \leq 0$. 
Since $e_\alpha(m_2) < 0$, assertion (ii) gives $e_\alpha (\theta^N m_2) < \theta^N e_\alpha (m_2)$ for each $\theta > 1$. 
By writing $m_1 + m_2 = \frac{m_1+m_2}{m_2} \cdot m_2 =: \theta_1^N m_2$ and $m_2 = \frac{m_2}{m_1} \cdot m_1 = : \theta_2^N m_1$, 
and noting $\theta_1,\theta_2>1$ and $\theta_1^N -1 = m_1/m_2 = \theta_2^{-N}$, assertions (ii) and (i) yield 
\[
e_\alpha(m_1+m_2) < \theta_1^N e_\alpha (m_2) = e_\alpha(m_2) + \theta_2^{-N} e_\alpha (\theta_2^N m_2) \leq e_\alpha(m_1) + e_\alpha(m_2). 
\]
Thus, the desired inequality is obtained. 
\end{proof}

\begin{proposition}\label{p:min-sign-Lag}
	Let $N \geq 1 $, $q>2$, $p_2 < p < p_q $, $m>0$ and $\alpha \geq \alpha_0(m)$. 	
	Then any minimizer corresponding to $e_\alpha(m)$ is either positive or negative in $\RN$. 
	Furthermore, any associated Lagrange multiplier to any minimizer is positive. 
\end{proposition}

\begin{proof}
	Let $u$ be any minimizer corresponding to $e_\alpha(m)$. 
	Since $E_\alpha(\vab{u}) = E_\alpha(u)$ and $\vab{u} \in S_m$, 
	$\vab{u}$ is also a minimizer. Notice that there exists $\lambda \in \R$ such that 
	$(\lambda, \vab{u} )$ is a solution of \eqref{eq_main}. 
	By \cref{l:reg}, $\vab{u} \in C^2(\RN)$. 
	Therefore, $\vab{u}$ is a solution of 
	\[
	- \sum_{i,j=1}^N b_{\vab{u}, ij} (x)  \frac{\partial^2 v}{\partial x_i \partial x_j} = - \lambda v + \alpha \vab{v}^{p-2} v \quad \text{in} \ \RN,
	\]
	where
	\[
	b_{\vab{u} , ij} (x) \coloneq \ab( 1 + \vab{ \nabla \vab{u} (x) }^{q-2}  ) \delta_{ij} 
	+ (q-2) \vab{ \nabla \vab{u} (x) }^{q-2} \frac{\partial_i \vab{u}(x)}{\vab{\nabla \vab{u}(x)}} 
	\cdot \frac{\partial_j \vab{u} (x)}{ \vab{\nabla \vab{u} (x)} } \in C(\RN). 
	\]
	Since $\vab{u} \not \equiv 0$, the strong maximum principle implies $\vab{u} > 0$ in $\RN$. 
	Therefore, $u$ does not change its sign and either $u>0$ or $u<0$ in $\RN$.

	About the positivity of Lagrange multipliers, 
	according to \cref{lem_Poho}, $u$ satisfies the Pohozaev identity: 
	\[
	0 = \frac{N-2}{2N} \Vab{\nabla u}_2^2 + \frac{N-q}{Nq} \Vab{\nabla u}_q^q + \frac{\lambda}{2} \Vab{u}_2^2 - \frac{\alpha}{p} \Vab{u}_p^p
	= E_\alpha(u) - \frac{1}{N} K(u) + \frac{\lambda}{2} \Vab{u}_2^2.
	\]
	Since $u \in S_m$ and $E_\alpha (u) = e_\alpha(m) \leq0$, the Pohozaev identity is rewritten as 
	\[
	0 = e_\alpha (m) - \frac{1}{N} K(u)  + \frac{\lambda}{2} m. 
	\]
	By $u \not \equiv 0$, $K(u) > 0$ and $\lambda > 0$ hold, which completes the proof. 
\end{proof}

\begin{proof}[Proof of \cref{theomin}]
	Notice that \cref{theomin} (ii) and (iii) follow from \cref{prop:min}. 
	Furthermore, the sign property of minimizers and the associated Lagrange multipliers hold by \cref{p:min-sign-Lag}. 
	If $u_0 \in S_m$ is a minimizer, then as in the proof of \cref{prop:min}, the symmetric decreasing rearrangement of $u_0$ is a radial minimizer 
	corresponding to $e_\alpha(m)$. 
	Therefore, it is enough to prove assertion (i).

	Recall $e_\alpha(m) < 0$ by \cref{prop:min} and $\alpha > \alpha_0(m)$. 
	Let $(u_n)_{n \in \N} \subset S_m$ satisfy $E_\alpha (u_n) \to e_\alpha (m) < 0$. 
	From \cref{lemma:I} (ii), $(u_n)_{n \in \N}$ is bounded in $X$. 
	
	\smallskip 
	
	\noindent
	\textbf{Claim 1:}
	\textsl{there exists $(x_n)_{n \in \N} \subset \RN$ such that }
	\begin{equation}\label{eq:nonva}
		\liminf_{n \to \infty} \Vab{u_n}_{L^2(B_1(x_n))} > 0.
	\end{equation}
	
	\smallskip 
	
	If \eqref{eq:nonva} does not hold, then $\sup_{z \in \RN} \Vab{u_n}_{L^2(B_1(z))} \to 0$ as $n \to \infty$ 
	and Lions' lemma  \cite{Lions84II} with the boundedness of $(u_n)_{n \in \N}$ in $X$ yields 
	$\Vab{u_n}_r \to 0$ as $n \to \infty$ for each $r \in (2,2^*)$. 
	Thus, \cref{lemma:I} (i) asserts $\Vab{u_n}_p \to 0$ and 
	\[
	0> e_\alpha (m) = \liminf_{n \to \infty} E_\alpha(u_n) \geq 0,
	\]
	which is a contradiction. Hence, \eqref{eq:nonva} holds.

	Since $E_\alpha$ is translation invariant, 
	for the relative compactness of $(u_n)_{n \in \N}$ up to translations, 
	we may assume that $(u_n)_{n \in \N}$ fulfills \eqref{eq:nonva} with $x_n = 0$. 
	Let $u_n \rightharpoonup u_\infty$ weakly in $X$. 
	From \eqref{eq:nonva} with $x_n=0$, $u_\infty \not \equiv 0$ and $0 < \Vab{u_\infty}_2^2 \leq m$.
	
	\smallskip 
	
	\noindent
	\textbf{Claim 2:}
	\textsl{prove $\Vab{u_\infty}_2^2 = m$.}
	
	\smallskip 
	
	To this end, suppose that $ 0 < \Vab{u_\infty}_2^2 < m$. 
	By taking a subsequence if necessary (still denoted by $(u_n)_{n \in \N}$), 
	we may suppose $u_n \to u_\infty$ in $L^2_{\rm loc} (\RN)$. 
	Furthermore, there exists $(R_n)_{n \in \N}$ such that
	\begin{equation}\label{eq:R_n}
		R_n \to \infty, \quad 
		\Vab{u_n - u_\infty}_{L^2(B_{3R_n})} \to 0, \quad 
		\Vab{u_n - u_\infty}_{L^p(B_{3R_n})} \to 0.
	\end{equation}
	In particular, the boundedness of $(u_n)_{n \in \N}$ in $X$ and \eqref{embeddings} lead to 
	\begin{equation}\label{e-vani-3Rn-Rn}
		\Vab{u_n}_{L^r \ab( B_{3R_n} \setminus B_{R_n} )} \to 0 \quad \text{for any $r \in [2,q^*)$}. 
	\end{equation}
	Select $(\varphi_n)_{n \in \N}, (\psi_n)_{n \in \N} \subset  C^\infty(\RN)$ such that 
	\begin{equation}\label{eq:phi_npsi_n}
		\begin{dcases}
			0 \leq \varphi_n \leq 1, \quad \varphi_n \equiv 1 \quad \text{on} \ B_{R_n}, \quad \varphi_n \equiv 0 \quad \text{on} \ \RN \setminus B_{2R_n}, \quad 
			\Vab{\nabla \varphi_n}_\infty \to 0,\\
			0 \leq \psi_n \leq 1, \quad \psi_n \equiv 0 \quad \text{on} \ B_{2R_n}, \quad \psi_n \equiv 1 \quad \text{on} \ \RN \setminus B_{3R_n}, \quad 
			\Vab{\nabla \psi_n}_\infty \to 0,
		\end{dcases}
	\end{equation}
	and set 
	\[ v_n \coloneq \varphi_n u_n, \quad w_n \coloneq \psi_n u_n.
	\]
	It is easily seen that 
	\begin{equation}\label{eq:convv_n}
		\Vab{v_n}_2^2 \to \Vab{u_\infty}_2^2 \in (0,m), \quad \Vab{w_n}_2^2 \to m - \Vab{u_\infty}_2^2 \in (0,m), 
		\quad \Vab{v_n - u_\infty}_p \to 0.
	\end{equation}
	Moreover, since there is $C>0$ such that 
	\[
	\vab{\xi + \eta}^q \leq \vab{\xi}^q + 2 \vab{\eta}^q + C \ab( \vab{\xi}^{q-1} \vab{\eta} + \vab{\xi} \vab{\eta}^{q-1} ) 
	\quad \text{for every $\xi,\eta \in \RN$},
	\]
	we see from \eqref{e-vani-3Rn-Rn} and \eqref{eq:phi_npsi_n} that 
	\[
	\begin{aligned}
		\int_{\RN} \vab{\nabla v_n}^q \odif{x} 
		&= 
		\int_{\RN} \vab{ \varphi_n \nabla u_n + u_n \nabla \varphi_n }^q \odif{x} 
		\\
		&\leq 
		\int_{\RN} \varphi_n^q \vab{\nabla u_n}^q \odif{x} + 2 \int_{\RN} \vab{\nabla \varphi_n}^q \vab{u_n}^q \odif{x} 
		\\
		&\quad 
		+C \int_{\RN} \varphi_n^{q-1} \vab{\nabla u_n}^{q-1} \vab{u_n} \vab{\nabla \varphi_n} + \varphi_n \vab{\nabla u_n} \vab{u_n}^{q-1} \vab{\nabla \varphi_n}^{q-1}  \odif{x}
		\\
		&= \int_{B_{2R_n}} \vab{\nabla u_n}^q \odif{x} + o_n(1). 
	\end{aligned}
	\]
	In a similar way, we may prove 
	\[
	\int_{\RN} \vab{\nabla w_n}^q \odif{x} = \int_{\RN} \vab{ \psi_n \nabla u_n + u_n \nabla \psi_n }^q \odif{x} 
	\leq \int_{\RN \setminus B_{2R_n}} \vab{\nabla u_n}^q \odif{x} + o_n(1). 
	\]
	Therefore, 
	\[
	\Vab{\nabla u_n}_q^q \geq \Vab{\nabla v_n}_q^q + \Vab{\nabla w_n}_q^q - o_n(1). 
	\]

	On the other hand, it is easily seen from \eqref{eq:R_n} and \eqref{eq:phi_npsi_n} that 
	\[
	\Vab{\nabla u_n}_2^2 \geq \Vab{\nabla v_n}_2^2 + \Vab{\nabla w_n}_2^2 - o_n(1), \quad 
	\Vab{u_n}_p^p = \Vab{v_n}_p^p + \Vab{w_n}_p^p + o_n(1).
	\]
	Thus, 
	\begin{equation}\label{eq:split}
		e_\alpha(m) = \lim_{n \to \infty} E_\alpha(u_n) \geq \liminf_{n \to \infty} E_\alpha(v_n) + \liminf_{n \to \infty} E_\alpha(w_n). 
	\end{equation}
	By \eqref{eq:convv_n} and the weak lower semicontinuity of norms, we also deduce that 
	\begin{equation}\label{eq:lowv_n}
		v_n \rightharpoonup u_\infty \quad \text{weakly in $X$}, \quad 
		\liminf_{n \to \infty} E_\alpha(v_n) \geq E_\alpha(u_\infty).
	\end{equation}
	Since $E_\alpha' \colon X \to X^*$ is a bounded map, 
	the facts that $0 < \Vab{u_\infty}_2^2 < m$, $\sqrt{m - \Vab{u_\infty}_2^2} / \Vab{w_n}_2^2 \to 1$ and $(w_n)_{n \in \N}$ is bounded in $X$ 
	show that 
	\[
	\liminf_{n \to \infty} E_\alpha(w_n) = \liminf_{n \to \infty} E_\alpha \ab( \frac{ \sqrt{m-\Vab{u_\infty}_2^2} }{\Vab{w_n}_2} w_n  ) 
	\geq 
	e_\alpha \ab( m - \Vab{u_\infty}_2^2 ).
	\]
	It follows from \eqref{eq:split}, \eqref{eq:lowv_n} and \eqref{eq:subadd} that 
	\[
	\begin{aligned}
		e_\alpha(m)= \lim_{n \to \infty} E_\alpha(u_n) 
		&\geq 
		\liminf_{n \to \infty} E_\alpha(v_n) + \liminf_{n \to \infty} E_\alpha(w_n)
		\\
		& \geq E_\alpha(u_\infty) + e_\alpha \ab( m - \Vab{u_\infty}_2^2 )
		\\
		&\geq 
		e_\alpha \ab( \Vab{u_\infty}_2^2 ) + e_\alpha \ab( m - \Vab{u_\infty}_2^2 ) \geq e_\alpha(m),
	\end{aligned}
	\]
	which implies that 
	\[ e_\alpha(m) = e_\alpha \ab( \Vab{u_\infty}_2^2 ) + e_\alpha \ab( m - \Vab{u_\infty}_2^2 ), 
	\quad e_\alpha\ab(\Vab{u_\infty}_2^2) = E_\alpha(u_\infty) =\liminf_{n \to \infty} E_\alpha(v_n).
	\]
	Since $e_\alpha(\Vab{u_\infty}_2^2)$ is achieved by $u_\infty$, 
	we use \cref{propem} (iii) with $m_1 = \Vab{u_\infty}_2^2$ and $m_2 = m - \Vab{u_\infty}_2^2$ to obtain the following contradiction: 
	\[
	e_\alpha \ab( \Vab{u_\infty}_2^2 ) + e_\alpha \ab( m - \Vab{u_\infty}_2^2 ) = e_\alpha(m) = e_\alpha (m_1+m_2) 
	< 
	e_\alpha \ab( \Vab{u_\infty}_2^2 ) + e_\alpha \ab( m - \Vab{u_\infty}_2^2 ). 
	\]
	Therefore, the case $0<\Vab{u_\infty}_2^2 < m$ does not happen, concluding Claim 2. 
	
	\smallskip 
	
	\noindent
	\textbf{Claim 3:}
	\textsl{Conclusion.}
	
	\smallskip 
	
	Since $u_n \rightharpoonup u_\infty$ weakly in $L^2(\RN)$, the fact $\Vab{u_\infty}_2^2= m$ yields $u_n \to u_\infty$ strongly in $L^2(\RN)$, 
	and $u_n \to u_\infty$ strongly in $L^p(\RN)$. From 
	\[
	e_\alpha (m) = \lim_{n \to \infty} E_\alpha (u_n) 
	\geq 
	\frac{1}{2} \liminf_{n \to \infty} \Vab{\nabla u_n}_2^2 + \frac{1}{q} \liminf_{n \to \infty} \Vab{\nabla u_n}_q^q - \frac{\alpha}{p} \Vab{u_\infty}^p 
	\geq E_\alpha(u_\infty) \geq e_\alpha (m),
	\]
	it follows that $\Vab{\nabla u_n}_2^2 \to \Vab{\nabla u_\infty}_2^2$, $\Vab{\nabla u_n}_q^q \to \Vab{\nabla u_\infty}_q^q$ and $E_\alpha(u_\infty) = e_\alpha(m)$. 
	Therefore, $\Vab{u_n - u_\infty}_X \to 0$ and $u_\infty$ is a minimizer, which completes the proof. 
\end{proof}

\section{Local minimizers with positive energy}\label{Secloc}

Throughout this section, let $\alpha_0 = \alpha_0 (m) > 0$ be the number in \cref{prop:min}. 
The first step consists of proving the following

\begin{lemma}\label{lemma:continuity}
  Let $N \geq 1 $, $q>2$, $p_2 < p < p_q$, $m>0$ and $\alpha \in (0,\alpha_0]$. 
  For any $m' \in (0,m]$, we also consider 
  \[
  \ov{e}_\alpha(m') \coloneq \inf \Set{ E_\alpha(u) | u \in S^{\rho_0}_{m'} },
  \]
  where $\rho_0 = \rho_0(m,\alpha_0) > 0$ is the number in \cref{lemma:I} (iii) and $S^{\rho_0}_{m'}$ is defined in \eqref{def-Smrho}. 
  Then the function $(0,m] \ni m' \mapsto \overline{e}_\alpha(m')$ is continuous and nonincreasing.
\end{lemma}

\begin{proof}
Let us start with the continuity of $\ov{e}_\alpha(\cdot)$.
It is sufficient to show that for any positive sequence $(m_k)_{k \in \N}$ such that $m_k \to m' \in (0,m]$ as $k \to \infty$ 
one has $\lim_{k \to \infty} \overline{e}_{\alpha}(m_k) = \overline{e}_\alpha (m')$. We first prove
  \begin{equation}\label{eq:upper}
    \limsup_{k \to \infty} \overline{e}_\alpha (m_k) \leq \overline{e}_\alpha(m').
  \end{equation}
For any $u \in S^{\rho_0}_{m'}$ and each $k \in \mathbb{N}^+$, set $u_k \coloneq \sqrt{m_k/m'} \cdot u \in S_{m_k}$. 
Since $u_k \to u$ in $X$, it is clear that $u_k \in S^{\rho_0}_{m_k}$ for any $k$ large enough and $\lim_{k \to \infty} E_\alpha(u_k) = E_\alpha(u)$. 
Thus
  \begin{equation*}
    \limsup_{k \to \infty} \overline{e}_\alpha(m_k) \leq \limsup_{k \to \infty} E_\alpha(u_k) = E_\alpha(u).
  \end{equation*}
By the arbitrariness of $u \in S^{\rho_0}_{m'}$, we conclude that \eqref{eq:upper} holds. 

To complete the proof of the continuity, it remains to show
  \begin{equation}\label{eq:lower}
    \liminf_{k \to \infty} \overline{e}_\alpha (m_k) \geq \overline{e}_\alpha (m').
  \end{equation}
For each $k \in \mathbb{N}^+$, there exists $v_k \in S^{\rho_0}_{m_k}$ such that
  \begin{equation}\label{eq:lower-1}
    E_\alpha(v_k) \leq \overline{e}_\alpha (m_k) + \frac{1}{k}.
  \end{equation}
Setting $t_k \coloneq \sqrt{m'/m_k}$, we have $\tilde{v}_k \coloneq t^{(2 - N)/2}_k v_k(\cdot/t_k) \in S^{\rho_0}_{m'}$ and thus
  \begin{equation*}
    \begin{split}
      \overline{e}_\alpha(m') \leq E_\alpha(\tilde{v}_k)
        & \leq E_\alpha (v_k) + \vab{E_\alpha (\tilde{v}_k) - E_\alpha (v_k) } \\
        & \leq \overline{e}_\alpha ({m_k}) + \frac{1}{k} + \vab{E_\alpha (\tilde{v}_k) - E_\alpha (v_k) } \\
        & \leq \overline{e}_\alpha (m_k) + \frac{1}{k} + \vab{ \ab(t_k^{N(1-q/2)}-1) \Vab{\nabla v_k}_q^q - \frac{\alpha}{p} \ab(t_k^{p(1-N/2)+N}-1)\Vab{v_k}_p^p} .
    \end{split}
  \end{equation*}
Since $t_k \to 1$ the proof of \eqref{eq:lower} can be reduced to show that $(v_k)_{k \in \N}$ is bounded in $X$. 
To justify the boundedness, by \eqref{eq:lower-1} and \eqref{eq:upper}, we have $\limsup_{k \to \infty} E_\alpha(v_k) \leq \overline{e}_\alpha(m')$. 
Noting that $v_k \in S_{m_k}$ and $m_k \to m$, it follows from \cref{lemma:I} (ii) that $(v_k)_{k \in \N}$ is bounded in $X$ 
and the continuity of $\overline{e}_\alpha(\cdot)$ is proved.

Now we prove the monotonicity of $\ov{e}_\alpha(\cdot)$, 
which is equivalent to showing that for any $0 < m_1 < m_2 \leq m$ and for an arbitrary $\varepsilon > 0$ one has
  \begin{equation}\label{eq:nonincrease_1}
    \overline{e}_\alpha(m_2) \leq \overline{e}_{\alpha}(m_1) + \varepsilon.
  \end{equation}
By the definition of $\overline{e}_{\alpha}(m_1)$, there exists $u \in S^{\rho_0}_{m_1}$ such that
  \begin{equation}\label{eq:nonincrease_2}
    E_\alpha(u) \leq \overline{e}_\alpha (m_1) + \frac{\varepsilon}{2}.
  \end{equation}
Let $\chi\in C^\infty_c(\mathbb{R}^N)$ be a radial cut-off function such that $\chi(x) = 1$ for $|x| \leq 1$ and $\chi(x) = 0$ for $|x| \geq 2$. 
For any $\delta > 0$, we set $u_\delta(x) \coloneq u(x) \chi(\delta x)$. 
Since $u_\delta \to u$ in $X$ as $\delta \to 0^+$, one can fix a small enough constant $\delta > 0$ 
such that $K(u_\delta) > \rho_0(\alpha_0,m)/2$ and 
  \begin{equation}\label{eq:nonincrease_3}
    E_\alpha (u_\delta) \leq E_\alpha (u) + \frac{\varepsilon}{4}.
  \end{equation}
Then take $v \in C^\infty_c(\mathbb{R}^N) \setminus \{0\}$ such that $\supp  (v) \subset B(0,1+4/\delta) \setminus B(0,4/\delta)$ 
and by $\Vab{u_\delta}_2^2 \leq \Vab{u}_2^2 = m_1 < m_2$,  set
\begin{equation*}
    \tilde{v} \coloneq \frac{\sqrt{m_2 - \Vab{u_\delta}^2_2}}{\Vab{v}_2} v.
  \end{equation*}
For any $\mu \leq 0$, we define $w_\mu \coloneq u_\delta + \mu \ast \tilde{v}$, where $\mu \ast \tilde{v} \coloneq e^{\mu N/2}\tilde{v}(e^{\mu} x)$. 
Since $u_\delta$ and  $\mu \ast \tilde{v}$ have disjoint support, 
it is clear that $w_\mu \in S^{\rho_0}_{m_2}$.  Noting that $K( \mu \ast \tilde{v} ) \to 0$ as $\mu \to - \infty$, it follows from \cref{lemma:I} (i) that
  \begin{equation}\label{eq:nonincrease_4}
    E_\alpha (\mu_0 \ast \tilde{v}) \leq \frac{\varepsilon}{4} \qquad \text{for some}~\mu_0 < 0.
  \end{equation}
Now, by the definition of $\overline{e}_\alpha(m_2)$, \eqref{eq:nonincrease_4}, \eqref{eq:nonincrease_3} and \eqref{eq:nonincrease_2}, we obtain
  \begin{equation*}
    \overline{e}_\alpha(m_2) 
    \leq 
    E_\alpha (w_{\mu_0}) 
    = 
    E_\alpha (u_\delta) + E_\alpha (\mu_0 \ast \tilde{v}) 
    \leq 
    E_\alpha (u) + \frac{\varepsilon}{2} \leq \overline{e}_{\alpha}(m_1) + \varepsilon,
  \end{equation*}
that is \eqref{eq:nonincrease_1}. 
\end{proof}

\begin{lemma}\label{l:no-critical}
	Let $N \geq 1$, $q > 2$, $p_2 < p < p_q$, $\alpha \in(0,\alpha_0]$ and $0<m' \leq m$. 
	Then the functional $E|_{S_{m'}}$ does not admit any critical point 
	in $\Set{ u \in S_{m'} | K(u) \leq \rho_0 }$. 
\end{lemma}

\begin{proof}
Let $u \in S_{m'}$ and $K(u) \leq \rho_0$. Since $m' \leq m$, \cref{lemma:I} (iii) yields $Q_\alpha(u) > 0$. 
On the other hand, if $u$ is a critical point of $E|_{S_{m'}}$, 
then there exists $\lambda \in \R$ such that $(\lambda,u)$ satisfies \eqref{eq_main} and 
$P_{\alpha,\lambda} (u) = 0$ due to \cref{lem_Poho}. 
Since 
\[
\begin{aligned}
	0 = 
	-P_{\alpha, \lambda} (u) 
	= - P_{\alpha,\lambda} (u) + \frac{N}{2} \Bab{ \Vab{\nabla u}_2^2 + \Vab{\nabla u}_q^q + \lambda \Vab{u}_2^2 - \alpha \Vab{u}_p^p } 
	= Q_{\alpha} (u),
\end{aligned}
\]
we deduce that there is no critical point of $E|_{S_{m'}}$ in $\Set{ u \in S_{m'} | K(u) \leq \rho_0 }$. 
\end{proof}

In view of \cref{theomin} and \cref{prop:min} (ii), 
there is a radial minimizer $u_0 \in S_m$ corresponding to $e_{\alpha_0} (m) = 0$. 
From \cref{l:no-critical}, $K(u_0) > \rho (\alpha_0,m)$ holds. 
In particular, $u_0 \in S_m^{\rho_0}$ and $\overline{e}_{\alpha_0} (m) = E_{\alpha_0} (u_0) = 0$. 
Since $\alpha \mapsto E_\alpha(u_0)$ is continuous, there is $\alpha_0' \in (0,\alpha_0)$ such that 
\begin{equation}\label{bdd-ebar}
	\ov{e}_{\alpha} (m) \leq E_{\alpha} (u_0) < \min \Set{ \frac{1}{8q} , \ \frac{1}{N} } \rho_0
	\quad \text{for each $\alpha \in (\alpha_0', \alpha_0]$}. 
\end{equation}

\begin{lemma}\label{l:pos-Lag}
		Let $N \geq 1$, $q > 2$, $p_2 < p < p_q$, $\alpha \in(\alpha_0',\alpha_0]$ and $0<m' \leq m$. 
	Suppose that 
	$ \ov{e}_\alpha (m') = \ov{e}_\alpha (m)$, 		
	$u \in S^{\rho_0}_{m'}$ is a minimizer for $\ov{e}_\alpha(m')$ and 
	$(\lambda,u) \in \R \times X$ satisfies \eqref{eq_main}. 
	Then $\lambda > 0$ and $m'=m$ hold. 
\end{lemma}

\begin{proof}
For $0<t \ll 1$, notice that $(1-t) u \in S^{ \rho_0 }_{ (1-t)^2m' }$. Thus \cref{lemma:continuity} leads to 
\[
0 
\leq 
\ov{e}_\alpha \ab( (1-t)^2 m' ) - \ov{e}_\alpha (m') 
\leq 
E_\alpha \ab( (1-t) u ) - E_\alpha (u). 
\]
Dividing by $t$ and letting $t \to 0^+$ give 
\[
0 \leq \odv*[delims-eval=.|]{ E_\alpha \ab( (1-t) u ) }{t}_{t=0} 
= - \Vab{\nabla u}_2^2 - \Vab{\nabla u}_q^q + \alpha \Vab{u}_p^p 
= \lambda \Vab{u}_2^2 .
\]
Thus, $\lambda \geq 0$ holds. 
If $\lambda =0$, then  
$\ov{e}_\alpha(m') = \ov{e}_\alpha(m)$, 
\eqref{bdd-ebar} and \cref{lem_Poho} yield 
\[
\frac{1}{N} \rho_0 > \ov{e}_\alpha(m) = \ov{e}_{\alpha} (m') 
= 
E_\alpha (u) = E_\alpha(u) - \frac{1}{N} P_{\alpha,0} (u) = \frac{1}{N} K(u). 
\]
However, this contradicts \cref{l:no-critical}. Therefore, $\lambda > 0$ holds.

We next prove $m = m'$. Argue by contradiction and assume $m' < m$. 
For $0<\e \ll 1$, since $(1+\e)^2 m' < m$ and $(1+\e) u \in S^{\rho_0}_{(1+\e)^2 m'}$, we have 
\[
\overline{e}_{\alpha} \ab( (1+\e) m' ) - \overline{e}_\alpha(m')  \leq E_\alpha \ab( (1+\e) u )  - E_\alpha(u) 
= \e  E_\alpha'(u)u + o(\e) = - \lambda \e \Vab{u}_2^2 + o(\e). 
\]
From $\lambda > 0$, it follows that $\ov{e}_\alpha ( \ab(1+\e) m' ) < \ov{e}_\alpha(m')$ holds for sufficiently small $\e \in (0,1)$. 
By \cref{lemma:continuity}, we infer that $ \ov{e}_\alpha (m) \leq \ov{e}_\alpha ((1+\e) m') < \ov{e}_\alpha (m') = \ov{e}_\alpha (m)$, which is absurd. 
Therefore, $m'=m$ holds.  
\end{proof}

\begin{proposition}\label{p:ex-loc-min}
	Suppose $N \geq 1$, $q > 2$, $p_2 < p < p_q$ and $\alpha \in (\alpha_0',\alpha_0)$. 
	Then $\ov{e}_\alpha (m)$ is attained by $v_\infty \in S_m^{\rho_0}$. 
	 Moreover, the associated Lagrange multiplier to $v_\infty$ is strictly positive 
	and the function $(\alpha_0',\alpha_0) \ni \alpha \mapsto \overline{e}_\alpha(m)$ is strictly decreasing. 
\end{proposition}

\begin{proof}
Let $\alpha \in (\alpha_0',\alpha_0)$ and $(u_n)_{n \in \N} \subset S^{\rho_0}_m$ be any minizming sequence 
for $\ov{e}_\alpha(m)$. 
From \cref{lemma:I} (ii), $(u_n)_{n \in \N}$ is bounded in $X$. Furthermore, 
by \eqref{Eest}, \eqref{bdd-ebar}, $K(u_n) > \rho_0/2$ and $E_\alpha(u_n) \to \ov{e}_\alpha(m)$, 
we may assume that for all $n \geq 1$, 
\begin{equation}\label{lowbddKu_n}
	\rho_0  < K(u_n);
\end{equation}
otherwise $K(u_n)/(2q) \leq E_\alpha (u_n) \leq \rho_0 / (4q)$ and this contradicts $K(u_n) > \rho_0/2$. 
Since $(u_n)_{n \in \N}$ is bounded in $X$, there exists $\delta_0 > 0$ such that 
\begin{equation}\label{uniawaybdry}
	\inf_{n \in \N} \dist_X \ab( \partial S_m^{\rho_0} , u_n ) \geq \delta_0 > 0, 
	\quad 
	\text{where} \ \partial S_m^{\rho_0} \coloneq \Set{ u \in S_m | K(u) = \frac{\rho_0}{2} }.
\end{equation}
To prove that $(u_n)_{n \in \N}$ contains a strongly convergent subsequence up to translations, 
we exploit another sequence which is close to $(u_n)_{n \in \N}$. Denote by $\overline{S^{\rho_0}_m}$ 
\[
\overline{S^{\rho_0}_m} \coloneq \Set{ u \in S_m | K(u) \geq \frac{\rho_0}{2} }. 
\]
Since $\overline{S^{\rho_0}_m}$ is closed in $X$ and $\ov{e}_\alpha(m) = \inf_{ \overline{S^{\rho_0}_m} } E_\alpha$, 
applying Ekeland's variational principle for $E_\alpha |_{ \overline{S^{\rho_0}_m} } $ and $(u_n)_{n \in \N}$ 
yields $(v_n)_{n \in \N} \subset \overline{S^{\rho_0}_m}$ such that 
\begin{equation}\label{Eke}
	\begin{aligned}
		&\ov{e}_\alpha (m) \leq E_\alpha(v_n) \leq E_\alpha(u_n), \quad \Vab{u_n - v_n}_X \to 0, 
		\\
		&
		E_\alpha(w) > E_\alpha(v_n) - o_n(1) \Vab{w-v_n}_X \quad \text{for each $w \in \overline{S^{\rho_0}_m}$ with $w \neq v_n$}.
	\end{aligned}
\end{equation}
Thus, $(v_n)_{n \in \N}$ is also a minimizing sequence for $\ov{e}_\alpha(m)$ and bounded in $X$. 
From \eqref{Eke}, \eqref{lowbddKu_n} and \eqref{uniawaybdry} it follows that 
\begin{equation}\label{lowbddv_n}
	\liminf_{n \to \infty} K(v_n) \geq \rho_0, \quad \Vab{E_\alpha'(v_n)}_{ T^*_{v_n} S_m } \to 0,
\end{equation}
where
\[
\begin{aligned}
	&\Vab{L}_{ T^*_uS_m } \coloneq \sup \Set{ Lw | w \in T_uS_m, \Vab{w}_X \leq 1 }, 
	\quad 
	T_uS_m \coloneq \Set{ w\in X | M'(u) w  = 0 }, 
	\\
	&M(u) \coloneq \Vab{u}^2_2.
\end{aligned}
\]
Notice that by \eqref{lowbddv_n}, we may find $(\lambda_n)_{ n \in \N} \subset \R$ such that 
\begin{equation}\label{vnPS}
	\Vab{E_\alpha'(v_n) + \lambda_n M'(v_n)}_{X^*} \to 0.
\end{equation}
From the boundeness of $(v_n)_{n \in \N}$, $(\lambda_n)_{n \in \N}$ is bounded and 
by passing to a subsequence if necessary, 
we may suppose $\lambda_n \to \lambda_\infty$ and $v_n \rightharpoonup v_\infty$ weakly in $X$. 
By \eqref{vnPS} and applying \cite[Theorem 2.1]{BoMu92} on $B_R$ for any $R>0$, we deduce that 
$\nabla v_n \to \nabla v_\infty$ strongly in $L^r_{\rm loc} (\RN)$ for any $r<q$. 
Therefore, $v_\infty$ is a solution of \eqref{eq_main} with $\lambda = \lambda_\infty$. 
Furthermore, up to subsequences, we may assume that $v_n \to v_\infty$ and $\nabla v_n \to \nabla v_\infty$ a.e. $\RN$.

In what follows, we will prove that $(v_n)_{n \in \N}$ has a strongly convergent subsequence in $X$ and 
the same is true for $(u_n)_{n \in \N}$ thanks to \eqref{Eke}.

	\smallskip 

\noindent
\textbf{Claim 1:}
\textsl{there exists $(x_n)_{n \in \N} \subset \RN$ such that }
\begin{equation}\label{nonva-loc}
	\liminf_{n \to \infty} \Vab{v_n}_{L^2(B_1(x_n))} > 0.
\end{equation}

\smallskip 

The argument is similar to the one in the proof of \cref{theomin}. 
If \eqref{nonva-loc} does not hold, then $\Vab{v_n}_p \to 0$ holds. 
By $v_n \in S_m^{\rho_0}$, \cref{lemma:I} (iii), \eqref{bdd-ebar} and \eqref{lowbddv_n}, the following contradiction occurs: 
\[
\frac{\rho_0}{8q} > \ov{e}_\alpha(m) 
= \liminf_{n \to \infty} E_\alpha(v_n) \geq \liminf_{n \to \infty} \frac{1}{q} K(v_n) \geq \frac{\rho_0}{q}. 
\]
Thus \eqref{nonva-loc} holds.

Since $E_\alpha$ is translation invariant, we may assume that $(v_n)_{n \in \N}$ fulfills \eqref{nonva-loc} with $x_n = 0$. 
From \eqref{nonva-loc}, $v_\infty \not \equiv 0$ and $m' \coloneq \Vab{v_\infty}_2^2 \in(0,  m ]$. 
Since $v_\infty$ is a critical point of $E_\alpha |_{S_{m'}}$, 
\cref{l:no-critical,l:pos-Lag} yields $K(v_\infty) > \rho_0$ and 
\begin{equation}\label{lowbdEv_inf}
	\overline{e}_\alpha \ab( m' ) \leq E_\alpha(v_\infty)
\end{equation}
By $v_n \to v_\infty$ and $\nabla v_n \to \nabla v_\infty$ a.e. $\RN$, the Brezis--Lieb Lemma implies 
\begin{equation}\label{split}
	E_\alpha(v_n) = E_\alpha(v_\infty) + E_\alpha (w_n) + o_n(1), \quad 
	K(v_n) = K(v_\infty) + K(w_n) + o_n(1),
\end{equation}
where $w_n \coloneq v_n - v_\infty$. Notice that $\Vab{w_n}_2^2 \to m - \Vab{v_\infty}_2^2 = m - m'$.

\smallskip 

\noindent
\textbf{Claim 2:}
\textsl{$\Vab{v_\infty}_2^2 = m$.}

\smallskip

The relations \eqref{mono} and $\alpha < \alpha_0(m)$ yield
\[
\liminf_{n \to \infty} E_\alpha (w_n) 
= \liminf_{n \to \infty} E_\alpha \ab(  \frac{ \sqrt{m-\Vab{v_\infty}_2^2} }{\Vab{w_n}_2} w_n  )
\geq e_\alpha \ab( m - \Vab{v_\infty}_2^2 ) \geq e_\alpha (m) = 0.
\]
Thus, \eqref{split}, \cref{lemma:continuity} and \eqref{lowbdEv_inf} give 
\[
\overline{e}_\alpha (m) 
\geq E_\alpha (v_\infty) + \liminf_{n \to \infty} E_\alpha (w_n) 
\geq E_\alpha (v_\infty) 
\geq \overline{e}_\alpha \ab( m' )
\geq \overline{e}_\alpha (m),
\]
which implies
\[
\overline{e}_\alpha(m) = E_\alpha(v_\infty) = \overline{e}_\alpha \ab( m' ). 
\] 
Thus, \cref{l:pos-Lag} leads to $\Vab{v_\infty}_2^2 = m$ and the Lagrange multiplier $\lambda_\infty$ is positive.

\smallskip 

\noindent
\textbf{Claim 3:}
\textsl{Conclusion.}

\smallskip 

From $\Vab{v_\infty}_2^2 = m$ and $v_n \rightharpoonup v_\infty$ weakly in $X$, 
it is easily seen that $\Vab{v_n - v_\infty}_2 \to 0$, which implies 
$\Vab{v_n - v_\infty}_r \to 0$ for any $r \in [2,q^*)$. The lower semicontinuity of norms with \eqref{lowbdEv_inf} and $m=m'$ leads to 
\[
\overline{e}_\alpha (m) \leq E_\alpha(v_\infty) \leq \liminf_{n \to \infty} E_\alpha(v_n) = \ov{e}_\alpha(m),
\]
which asserts $\Vab{\nabla v_n}_2^2 \to \Vab{\nabla v_\infty}_2^2$ and $\Vab{\nabla v_n}_q^q \to \Vab{\nabla v_\infty}_q^q$. 
Hence $\Vab{v_n - v_\infty}_X \to 0$ and $v_\infty$ is the desired minimizer for $\ov{e}_\alpha(m)$ and 
$(u_n)_{n \in \N}$ also converges in $X$ due to \eqref{Eke}.

Finally, by 
\[
\pdv*{E_\alpha (v_\infty)}{\alpha} = - \frac{1}{p} \Vab{v_\infty}_p^p < 0, \quad 
\ov{e}_{\alpha + \e} (m) \leq  E_{\alpha + \e} (v_\infty) < E_\alpha (v_\infty) = \ov{e}_\alpha(m),
\]
the function $(\alpha_0',\alpha_0) \ni \alpha \mapsto \ov{e}_\alpha(m)$ is strictly decreasing. 
\end{proof}

Finally, we prove

\begin{lemma}\label{lemma:sign}
	Let $N \geq 1$, $q>2$, $p_2 < p < p_q $, $\alpha \in (\alpha_0', \alpha_0)$. 
	Then any minimizer of $\overline{e}_\alpha(m)$ has a constant sign. 
	Furthermore, there exists a radial local minimizer of $\overline{e}_\alpha(m)$. 
\end{lemma}

\begin{proof}
	For given minimizer $v \in S^{\rho_0}_m$ of $\overline{e}_\alpha (m)$, we set
	\begin{equation*}
		v^+ \coloneq \max \set{0, v } \qquad \text{and} \qquad v^- \coloneq \min\set{0, v},
	\end{equation*}
	and suppose by contradiction that $m^+ \coloneq \Vab{v^+}^2_2 \neq 0$ and $m^- \coloneq \Vab{v^-}^2_2 \neq 0$.  
	Clearly, \cref{l:no-critical} gives
	\begin{equation*}
		\rho_0 < K(v) = K(v^+) + K(v^-) .
	\end{equation*}
	Without loss of generality, we may assume that $v^+ \in S^{\rho_0}_{m^+}$ 
	and thus $E_\alpha(v^+) \geq \overline{e}_{\alpha} (m^+) $. 
	By $\alpha < \alpha_0(m)$, \cref{prop:min}, \cref{Rem:alp0m} and \eqref{mono}, 
	the relation $\alpha < \alpha_0 (m) \leq \alpha_0 (m^-)$ holds 
	and the global infimum $e_\alpha(m^-) = 0$ is not achieved, and hence $E_\alpha(v^-) > e_\alpha (m^-) = 0$. 
	Using also \cref{lemma:continuity}, we obtain a contradiction
	\begin{equation*}
		\overline{e}_\alpha(m) = E_\alpha(v) = E_\alpha(v^+) + E_\alpha(v^-) > \overline{e}_\alpha (m^+) \geq \overline{e}_\alpha (m).
	\end{equation*}
	Hence $v$ has a constant sign on $\mathbb{R}^N$.

For the existence of radial local minimizers, let $u \in S^{\rho_0}_m$ be a local minimizer. We may suppose $u>0$ in $\RN$. 
Since $E_\alpha(u^*) \leq E_\alpha(u)$ holds where $u^*$ is the symmetric decreasing rearrangement of $u$,  
if we can prove $u^* \in S^{\rho_0}_m$, 
then $u^*$ is the desired radial local minimizer. 
To prove $u^* \in S^{\rho_0}_m$, notice that $(\lambda,u)$ is a solution of \eqref{eq_main} and 
that the Pohozaev identity $P_{\alpha,\lambda} (u) = 0$ holds thanks to \cref{lem_Poho}. Hence, 
\[
0 = \frac{N}{2} \ab( \Vab{\nabla u}_2^2 + \Vab{\nabla u}_q^q + \lambda \Vab{u}_2^2 - \alpha \Vab{u}_p^p ) - P_{\alpha,\lambda} (u) 
= Q_\alpha(u) \geq Q_\alpha(u^*). 
\]
Thus, \cref{lemma:I} (iii) yields $K(u^*) > \rho_0$ and $u^* \in S^{\rho_0}_m$.  
\end{proof}

\begin{proof}[Proof of \cref{theorem:local}]
	It is easily seen that \cref{theorem:local} (i) and (iii) hold by \cref{prop:min,p:ex-loc-min} and \cref{l:no-critical,l:pos-Lag,lemma:sign}.
	For (ii), what remains to prove is that local minimizers of $\ov{e}_\alpha(m)$ are an energy ground state solution. 
	However, this follows from \cref{l:no-critical}. Thus, \cref{theorem:local} (ii) also holds. 
\end{proof}

\section{Mountain pass type solution}\label{Secmou}

In the present section, 
we prove the existence of critical point $u$ of $E_\alpha |_{S_m}$ with $E_\alpha (u) > \ov{e}_\alpha (m)$ 
when $\alpha_0' < \alpha$ via the mountain pass theorem for $E_\alpha |_{S_m}$. 
For this purpose, let us introduce 
\[
X_\rad \coloneq \Set{ u \in X | \text{ $u$ is radially symmetric} }.
\]

To prove \cref{t:ex-mp}, we use the approach in \cite{Je97} and introduce an auxiliary functional $J$.
For any $u\in X_\rad$ and $\s\in\R$, we define 
$$(\s\ast u)(x):=e^{\sigma N/2}u(e^{\s} x)$$
and the  functional $J:\R\times X_\rad\to \R$  as 
\begin{equation}\label{Jtheta}
J(\s, u) \coloneq
E_\alpha \ab(e^{\sigma N/2}u\big(e^{\s}\cdot ))
=\frac{1}2 e^{2\s}\Vab{\nabla u}_2^2 + \frac{1}{q} e^{\s\left[q(\frac{N}{2}+1)-N\right]}\Vab{\nabla u}^q_q  -\frac{\alpha}{p}e^{\s N(\frac p2-1)}\|u\|_p^p.
\end{equation}
Notice that 
\[
2<N\left(\frac p2-1\right)<q\left(\frac{N}{2}+1\right)-N
\]
and $J(\sigma, u) \to 0$ as $\sigma \to - \infty$ for each $u \in X_\rad$. Write 
\[
S_{m,r} \coloneq \Set{ u \in S_m | u \in X_\rad }.
\]

In what follows, we fix $m>0$ and $\alpha \in (\alpha_0',\infty)$ 
where $\alpha_0' >0$ is chosen so that \eqref{bdd-ebar} holds. 
Let $\rho_0 = \rho_0 (\alpha,m) > 0$ be the number in \cref{lemma:I} (iii). 
In view of \cref{theomin} and \cref{theorem:local}, when $\alpha_0' < \alpha < \alpha_0$, 
$v_\alpha$ denotes a radial local minimizer of $\overline{e}_\alpha(m)$, 
and when $\alpha_0 \leq \alpha$, $v_\alpha$ is a radial global minimizer of $e_\alpha(m)$. 
As in the proof of \cref{lemma:sign}, $Q_\alpha (v_\alpha) = 0$ and 
\[
K(v_\alpha) > \rho_0.
\]
Recalling \cref{lemma:I} (iii), \eqref{bdd-ebar} and $e_\alpha (m) \leq 0$, we have 
\begin{equation}\label{lowbdEa1}
	E_\alpha (v_\alpha) \leq \overline{e}_\alpha (m)  < \min \Set{ \frac{1}{8q} , \frac{1}{N} } \rho_0 
	< \frac{1}{2} \inf_{\substack{ u \in S_{m,r} \\ K(u) = \rho_0}} E_\alpha(u) \eqcolon \frac{1}{2} \delta_0.
\end{equation}
On the other hand, since $K( \sigma * u ) \to 0$ as $\theta \to -\infty$, there exists $v_0 \in S_{m,r}$ such that 
\[
K(v_0) < \rho_0, \quad E_\alpha (v_0) < \frac{1}{2} \delta_0. 
\]
Therefore, 
\begin{equation}\label{gammamp}
	\Gamma \coloneq \Set{ \gamma\in C([0,1],S_{m,r}) | K(\gamma(0)) < \rho_0 < K(\gamma(1)), 
		\ E_\alpha (\gamma (0)) < \frac{\delta_0}{2}, \ E_\alpha(\gamma(1)) < \frac{\delta_0}{2}  } \neq \emptyset. 
\end{equation}
Thus, $E_\alpha |_{S,r}$ has a mountain pass geometry and 
we can define its mountain pass level by
\[ c_{mp} \coloneq \inf_{\gamma\in  \Gamma}\max_{t\in [0,1]}E_\alpha \ab(\gamma(t)).\]
Observe that for every $\gamma \in \Gamma$ there exists $t_\gamma \in (0,1)$ such that $K(\gamma(t_\gamma)) = \rho_0$. 
Thus, \eqref{lowbdEa1} implies
\begin{equation*}
	c_{mp} \geq \delta_0.
\end{equation*}
We define the following minimax level for $J|_{\R\times S_{m,r}}$:
\[
\wt{c}_{mp} \coloneq \inf_{(\s,\gamma)\in \Sigma\times \Gamma}\max_{t\in [0,1]}J\ab(\s(t),\gamma(t) ), 
\]
where $\Gamma$ is defined in \eqref{gammamp} and 
\[\Sigma \coloneq \Set{ \s\in C([0,1],\R) | \s(0)=\s(1)=0}.\]

As in \cite{Je97} (see also \cite[Lemma 5.4]{BaMePo25}), the following hold: 
\begin{lemma}
	The minimax levels of $E_\alpha$ and $J$ coincide, namely $c_{mp}=\wt{c}_{mp}$.
\end{lemma}

Next, as in \cite[Proposition 2.2]{Je97} (\cite[Lemma 5.4]{BaMePo25}), 
via Ekleand's variational principle, we have

\begin{lemma}
\label{le:ekeland}
	Let $\e>0$. Suppose that $(\tilde\s,\tilde\gamma)\in \Sigma \times\Gamma$ satisfies 
	\begin{equation*}
		\max_{t \in [0,1]}J\ab(\tilde\s(t),\tilde\gamma(t))\le c_{mp} +\e.
	\end{equation*}
	Then there exists $(\s, u)\in \R\times S_{m,r}$ such that
	\begin{enumerate}[label={\rm (\arabic{*})},ref=\arabic{*}]
		\item 
		 $J(\s,u)\in [c_{mp}-\e,c_{mp}+\e]$;
		\item 
		$\dist_{\R \times S_{m,r}} \ab( (\s,u), \tilde\s \ab( [0,1] ) \times \tilde\gamma \ab( [0,1] ) ) \le \sqrt{\e}$;
		\item 
		$\Vab{ \ab( J|_{\R\times S_{m,r}} )'(\s,u) } \leq 8 \sqrt{\e}$.
	\end{enumerate}
\end{lemma}

\begin{proof}[Proof of \cref{t:ex-mp}]
We proceed through steps.

\smallskip 

\noindent
\textbf{Step 1:}
\textsl{The construction of a special Palais Smale sequence associated to the auxiliary functional $J$ defined in \eqref{Jtheta}.}

\smallskip 

Choose $(\gamma_n)_{n \in \N} \subset \Gamma$ so that 
\[
\max_{t \in [0,1]} E_\alpha (\gamma_n(t)) \to c_{mp}. 
\]
Since $E_\alpha(\vab{u}) = E_\alpha(u)$ and $K(\vab{u}) = K(u)$, 
we have $\vab{\gamma_n} \in \Gamma$. 
Thus, replacing $\gamma_n(t)(x)$ by $\vab{\gamma_n(t)} (x)$, 
we may also suppose $\gamma_n(t) \geq 0$ for each $t \in [0,1]$. 
Furthermore, let $(\rho_\eta)_{\eta > 0}$ be a mollifier. It is not hard to verify 
$\max_{0 \leq t \leq 1} \Vab{ \rho_\eta * \gamma_n(t) - \gamma_n(t) }_X \to 0$ as $\eta \to 0$. 
By changing $\gamma_n(t)$ by $ \sqrt{m} (\rho_\eta * \gamma_n (t) ) / \Vab{ \rho_\eta * \gamma_n (t) }_2$, 
we may suppose $\gamma_n (t) \in C^\infty(\RN) \cap X_\rad$ for every $t \in [0,1]$. 
Finally, since $E_\alpha(u^*) \leq E_\alpha(u)$ for each $u \in X$ where $u^*$ denotes the symmetric rearrangement of $u$, 
by considering $(\gamma_n(t))^*$ and \cite{AlLi89,Co84}, 
the map $[0,1] \ni t \mapsto (\gamma_n(t))^* $ belongs to $\Gamma$ and 
we may suppose that $\gamma_n(t)$ is nonincreasing in $\vab{x}$ for any $t \in [0,1]$.

Since $(0,\gamma_n) \in \Sigma \times \Gamma$ satisfies 
\[
\max_{t \in [0,1]} J \ab( 0, \gamma_n(t) ) \to c_{mp},
\]
by \cref{le:ekeland}, there exist $ \ab(  (\sigma_n,u_n) )_{n \in \N} \in \R \times S_{m,r}$ such that 
\begin{enumerate}[label=(\alph*), ref=\alph{*}]
	\item 
	$J(\sigma_n,u_n) \to c_{mp} > 0$; 
	\item \label{PS2}
	$\vab{\sigma_n} + \dist_{S_{m,r}} \ab( u_n , \gamma_n([0,1]) ) \to 0$; 
	\item \label{PS3}
	$\Vab{ \ab( J|_{\R\times S_{m,r}} )'(\s_n,u_n) } \to 0$. 
\end{enumerate}
By $\sigma_n \to 0$ and $ E_\alpha ( \sigma_n * u_n ) = J(\sigma_n, u_n) \to c_{mp}$, 
\cref{lemma:I} implies that $(u_n)_{n \in \N}$ is bounded in $X$. 
Then $u_n \rightharpoonup u_\infty$ weakly in $X$. 
In addition, by \eqref{PS2}, we observe that $\dist_X(u_n, \gamma_n([0,1])) \to 0$ and $\Vab{u_n^-}_2 \to 0$ due to 
$\gamma_n (t) \geq 0$ and 
\[
\Vab{u_n^-}_2 \leq \min_{t \in [0,1]} \Vab{ u_n - \gamma_n(t) }_2 \leq \min_{t \in [0,1]} \Vab{u_n - \gamma_n(t)}_{X} \to 0.
\]
Hence $u_\infty \geq 0$ holds. 
Furthermore, from $\sigma_n \to 0$ and \eqref{PS3}, there exists $\lambda_n \in \R$ such that 
\begin{equation}\label{eq:PSfull}
	\Vab{\partial_u J(0,u_n) + \lambda_n \int_{\RN} u_n \cdot \odif{x} } 
	=
	\Vab{-\Delta u_n -\Delta_q u_n + \lambda_n u_n - \alpha \vab{u_n}^{p-2} u_n}_{X^*} \to 0.
\end{equation}
Since $(u_n)_{n \in \N}$ is bounded in $X$, so is $(\lambda_n)_{n \in \N}$ in $\R$. 
Let $\lambda_n \to \lambda_\infty$. 
As in the proof of the proof of \cref{p:ex-loc-min}, we may verify that 
$(\lambda_\infty,u_\infty)$ satisfies \eqref{eq_main} and $P_{\alpha,\lambda_\infty} (u_\infty) = 0$ holds due to \cref{lem_Poho}.

\smallskip 

\noindent
\textbf{Step 2:}
\textsl{$u_n \to u_\infty$ strongly in $X$.}

\smallskip

We first remark that $\Vab{u_n - u_\infty}_r \to 0$ holds for any $r \in (2,q^*)$. 
In fact, when $N \geq 2$, this follows from the Radial Lemma in \cite{Li82}. 
On the other hand, when $N=1$, since $\min_{0 \leq t \leq 1} \Vab{u_n - \gamma_n(t)}_X \to 0$, 
there is a $t_n \in [0,1]$ such that $\Vab{u_n - \gamma_n(t_n)}_X \to 0$. 
Recalling that $\gamma_n(t_n)$ is nonincreasing in $[0,\infty)$ and $u_n \rightharpoonup u_\infty$ weakly in $X$, 
we infer that $\gamma_n(t_n) \rightharpoonup u_\infty$ weakly in $X$ and 
$\Vab{\gamma_n(t_n) - u_\infty }_\infty \to 0$ (see the argument in the proof of Claim 2 of \cref{prop:min}). 
Thus, $\Vab{u_n - u_\infty}_r \to 0$ holds for any $r \in (2,\infty]$ by Sobolev's embedding and $\Vab{\gamma_n(t_n) - u_n}_X \to 0$.

We next prove $u_\infty \not \equiv 0$. If $u_\infty \equiv 0$, then $\Vab{u_n}_p \to 0$. 
From $\sigma_n \to 0$ and 
\[
\begin{aligned}
	o(1) 
	&= \partial_\sigma J(\sigma_n,u_n) 
	\\
	&= \Vab{\nabla u_n}_2^2 + \ab( \frac{N}{2} + 1 - \frac{N}{q} ) \Vab{\nabla u_n}_q^q 
	- \alpha  N\ab( \frac{1}{2} - \frac{1}{p} ) \Vab{u_n}_p^p + o(1),
\end{aligned}
\]
we infer that $\Vab{\nabla u_n}_2 + \Vab{\nabla u_n}_q \to 0$. Therefore, 
the following contradiction occurs: $0 < c_{mp} = \lim_{n \to \infty} J(\sigma_n,u_n) = 0$. 
This yields $u_\infty \not \equiv 0$ and $0<\Vab{u_\infty}_2^2 \leq \lim_{n\to\infty} \Vab{u_n}_2^2 = m$.

Under assumption \eqref{A}, since $(\lambda_\infty,u_\infty)$ satisfies \eqref{eq_main} with $u_\infty \not \equiv 0$, 
either $\lambda_\infty > 0$ or $\lambda_\infty < 0$. 
When $\lambda_\infty > 0$, it follows from \eqref{eq_main} with $\lambda = \lambda_\infty$ that 
\[
\Vab{\nabla u_\infty}_2^2 + \Vab{\nabla u_\infty}_q^q + \lambda_\infty \Vab{u_\infty}_2^2 = \alpha \Vab{u_\infty}_p^p. 
\]
On the other hand, \eqref{eq:PSfull}, $\lambda_\infty > 0$ and the lower semicontinuity of norms lead to 
\[
\begin{aligned}
	\Vab{\nabla u_\infty}_2^2 + \Vab{\nabla u_\infty}_q^q + \lambda_\infty \Vab{u_\infty}_2^2 
	&\leq \liminf_{n \to \infty} 
	\Bab{\Vab{\nabla u_n}_2^2 + \Vab{\nabla u_n}_q^q + \lambda_\infty \Vab{u_n}_2^2 }
	\\
	&\leq \limsup_{n \to \infty} 
	\Bab{\Vab{\nabla u_n}_2^2 + \Vab{\nabla u_n}_q^q + \lambda_\infty \Vab{u_n}_2^2 }
	\\
	&= \limsup_{n \to \infty} 
	\Bab{  \partial_u J(0,u_n) u_n + \lambda_n \Vab{u_n}_2^2 + \alpha \Vab{u_n}_p^p } 
	\\
	&= \alpha \Vab{u_\infty}_p^p 
	= \Vab{\nabla u_\infty}_2^2 + \Vab{\nabla u_\infty}_q^q + \lambda_\infty \Vab{u_\infty}_2^2.
\end{aligned}
\]
In particular, $\Vab{\nabla u_n}_2^2 \to \Vab{\nabla u_\infty}_2^2$, $\Vab{\nabla u_n}_q^q \to \Vab{\nabla u_\infty}_q^q$ 
and $\Vab{u_n}_2^2 \to \Vab{u_\infty}_2^2$. 
Therefore, $\Vab{u_n -u_\infty}_X \to 0$.

We next consider the case $\lambda_\infty < 0$ and follow the argument in \cite[Proof of Theorem 1.5]{MeSz25}. 
By $P_{\alpha, \lambda_\infty} (u_\infty) = 0$ and 
\[
\Vab{\nabla u_\infty}_2^2 + \Vab{\nabla u_\infty}_q^q + \lambda_\infty \Vab{u_\infty}_2^2 - \alpha \Vab{u_\infty}^p_p = 0, 
\]
we have 
\[
0 = \frac{N}{2} \ab( \Vab{\nabla u_\infty}_2^2 + \Vab{\nabla u_\infty}_q^q + \lambda_\infty \Vab{u_\infty}_2^2 - \alpha \Vab{u_\infty}_p^p ) - P_{\alpha,\lambda_\infty} (u_\infty) 
=  Q_{\alpha,\lambda_\infty} (u_\infty). 
\]
On the other hand, letting $n \to \infty$ in $o_n(1) = \partial_{\sigma} J(\sigma_n,u_n)$ and noting $\Vab{u_n}_p \to \Vab{u_\infty}_p$ 
with the lower semicontinuity of norms yield $Q_{\alpha,\lambda_\infty} (u_\infty) \leq 0$. 
In particular, we deduce that $\Vab{\nabla u_n}_2^2 \to \Vab{\nabla u_\infty}_2^2$ and 
$\Vab{\nabla u_n}_q^q \to \Vab{\nabla u_\infty}_q^q$. 
Hence, $\nabla u_n \to \nabla u_\infty$ strongly in $L^2(\RN) \cap L^q(\RN)$. 
Finally, $\lambda_\infty < 0$, \eqref{eq:PSfull} and the fact that $(\lambda_\infty,u_\infty)$ is a solution of \eqref{eq_main} give 
\[
-\lambda_\infty m = \lim_{n \to \infty} \ab( \Vab{\nabla u_n}_2^2 + \Vab{\nabla u_n}_q^q - \alpha \Vab{u_n}_p^p ) 
= \Vab{\nabla u_\infty}_2^2 + \Vab{\nabla u_\infty}_q^q - \alpha \Vab{u_\infty}_p^p
= -\lambda_\infty \Vab{u_\infty}_2^2. 
\]
Thus, $\Vab{u_\infty}_2^2 = m$ and this gives $\Vab{u_n - u_\infty}_2 \to 0$.

Finally, since $u_\infty \in X_\rad$ is a solution of \eqref{eq_main}, 
\cref{l:reg} implies $u \in C^2(\RN)$. 
From \eqref{PS2} in the above, $u_\infty \geq 0$ in $\RN$. 
By the strong maximum principle, we get $u_\infty > 0$ in $\RN$ and complete the proof. 
\end{proof}

%

\section{Liouville Theorems}\label{SecLio}

The present section is devoted to the proof of \cref{nonex1}. 
Before the proof, we seek for the radial form of \eqref{eq_zero_mass}. Let $u \in X_\rad$ be a solution of \eqref{eq_zero_mass}. 
From 
\begin{equation*}
	\int_{\RN} \nabla u \cdot \nabla \varphi + \vab{\nabla u}^{q-2} \nabla u \cdot \nabla \varphi \odif{x} 
	= \int_{\RN} \alpha \vab{u}^{p-2} u \varphi \odif{x} \quad \text{for each $\varphi \in C^\infty_{c,\rad}(\RN)$},
\end{equation*}
it follows that 
\[
\int_0^\infty r^{N-1} u' \varphi' + r^{N-1} \vab{u'}^{q-2} u' \varphi' \odif{r} 
= \alpha \int_0^\infty \vab{u}^{p-2} u \varphi r^{N-1} \odif{r}.
\]
Therefore, the radial form of \eqref{eq_zero_mass} is 
\begin{equation}\label{eq_rad}
	-\ab( r^{N-1} \ab( 1 + \vab{u'}^{q-2} ) u' )' = 
	-\ab( r^{N-1} u' )' - \ab( r^{N-1} \vab{u'}^{q-2} u' )' = \alpha r^{N-1} \vab{u}^{p-2} u.
\end{equation}
This is rewritten as 
\begin{equation}\label{eq_rad-ver2}
	- \ab( 1 + (q-1)\vab{u'}^{q-2} ) u''(r) - \frac{N-1}{r} \ab( 1 + \vab{u'(r)}^{q-2} ) u'(r)  = \alpha \vab{u}^{p-2} u.
\end{equation}

\begin{proof}[Proof of \cref{nonex1}]
	(i) 
We exploit the argument in \cite[Proof of Lemma 3.12]{BaSo17} and \cite[Proof of Proposition 3.3]{IkTa19}. 
Let $u \in X_\rad \setminus \set{0}$ be a solution of \eqref{eq_main}. Remark that $u \in L^\infty(\RN)$. Now we divide the proof into intermediate claims.

\smallskip 

\noindent
\textbf{Claim 1:} 
\textsl{The number of zeros of $u$ in $(1,\infty)$ is finite.}

\smallskip 

Notice that by \cref{l:reg}, $u$ can be regarded as a classical solution to the following uniformly elliptic equation: 
\[
-\sum_{i,j=1}^N a_{ij} (x) u_{ij} = \alpha \vab{u}^{p-2} u, \quad 
a_{ij} (x) \coloneq \ab( 1 + \vab{\nabla u(x)}^{q-2} ) \delta_{ij} + (q-2) \vab{\nabla u(x)}^{q-2} \frac{u_i(x)}{\vab{\nabla u(x)}} \cdot \frac{u_j(x)}{\vab{\nabla u(x)}}.
\]
In particular, if $\pm u>0$ in $(a,b)$ and $u(b) = 0$ for some $0 \leq a < b$, then Hopf's lemma gives $\mp u'(b) > 0$. 
Thus, each zero of $u$ in $(0,\infty)$ is isolated. 

To prove Claim 1, we assume by contradiction that there exists $(r_n)_{n \in \N}$ such that 
$1<r_1 < r_2 < \dots < r_n < r_{n+1} < \cdots $, $u(r_n) = 0$ and $u \not \equiv 0$ in $(r_n,r_{n+1})$. 
Set $A_n \coloneq \Set{ x \in \RN | r_n < \vab{x} < r_{n+1} }$ and 
$v_n \coloneq \bm{1}_{A_n} u$ where $\bm{1}_A$ denotes the characteristic function of $A \subset \RN$. 
It is easily seen that $v_n \in W^{1,q}_0(A_n) \subset H^1_0(A_n) $ due to $q>2$. 
Furthermore, using $v_n$ as a test function to \eqref{eq_zero_mass} yields 
\[
\int_{\RN} \vab{\nabla v_n}^2 + \vab{\nabla v_n}^q \odif{x} = 
\int_{A_n} \vab{\nabla u}^2 + \vab{\nabla u}^q \odif{x} = \alpha \int_{A_n} \vab{u}^p \odif{x} 
= \alpha \int_{\RN} \vab{v_n}^p \odif{x}.
\]

Choose any $r \in (p_2,  \min \Bab{p,2^*} )$. Then, by the Gagliardo--Nirenberg inequality
\[
\alpha \int_{\RN} \vab{v_n}^p \odif{x} \leq \alpha \Vab{u}_\infty^{p-r} \Vab{v_n}_{r}^{r} 
\leq 
C \Vab{\nabla v_n}_2^{ r \nu_{r,2} } \Vab{v_n}_2^{ r(1-\nu_{r,2}) } 
\leq 
C \Vab{\nabla v_n}_2^{ r \nu_{r,2} } \Vab{u}_2^{ r(1-\nu_{r,2}) }
\]
and $r \nu_{r,2} > 2$, we end up with
\[
\begin{aligned}
	\Vab{\nabla v_n}_2^2  \le  \int_{\RN} \vab{\nabla v_n}^2 + \vab{\nabla v_n}^q \odif{x} 
	= 
	\alpha \int_{\RN} \vab{v_n}^p \odif{x} 
	\leq 
	C \Vab{\nabla v_n}_2^{r \nu_{r,2}} \Vab{u}_2^{r(1-\nu_{r,2})}.
\end{aligned}
\]
Thus, there exists $c>0$ such that 
\[
0 < c \leq \Vab{\nabla v_n}_2^{ r \nu_{r,2} - 2 } \quad \text{for each $n \geq 1$}. 
\]
In particular, 
\[
\Vab{\nabla u}_2^2 = \sum_{n=1}^\infty \Vab{\nabla u}_{L^2(A_n)}^2 = \sum_{n=1}^\infty \Vab{\nabla v_n}_2^2 = \infty,
\]
which is a contradiction. Thus, Claim 1 holds.

\smallskip 

\noindent
\textbf{Claim 2:} 
\textsl{$u \not \in L^2(\RN)$ by the comparison principle.}

\smallskip 

By Claim 1 and changing $u$ to $-u$ if necessary, we may assume that $u > 0$ in $(R_0,\infty)$, where $R_0 > 1$. 
If $r_0 \in (R_0,\infty)$ satisfies $u'(r_0) = 0$, then 
\eqref{eq_rad} and the fact $q>2$ yield  
\[
0 < \alpha r^{N-1}_0 u(r_0)^{p-1} = - (r^{N-1} u'(r) )'|_{r=r_0} - (r^{N-1} \vab{u'}^{q-2} u' )' |_{r=r_0} 
= -r_0^{N-1} u''(r_0).
\]
Thus, in $(R_0,\infty)$, there exists at most one local maximum point of $u$ and there is no other critical points. 
Since $u(r) \to 0$ as $r \to \infty$, we may find $R_1 \geq R_0 > 1$ such that $u'<0$ in $(R_1,\infty)$.

Let us divide the proof according to the dimension of the space $N$. 

\smallskip 

\noindent
\textbf{The case $N=1,2$:}
For $\e \in (0,1)$, we consider 
\[
\varphi(r) \coloneq \frac{\e}{\log r}. 
\]
Then 
\[
\begin{aligned}
	\varphi'(r) &= - \e (\log r)^{-2} r^{-1}, 
	\\ 
	r^{N-1} \vab{\varphi'(r)}^{q-2} \varphi'(r) 
	&= - r^{N-1-(q-1)} \e^{q-1} \ab( \log r )^{-2(q-1)} 
	= - r^{N-q} \e^{q-1} \ab( \log r)^{-2(q-1)}
\end{aligned}
\]
and 
\[
\begin{aligned}
	&- \ab( r^{N-1} \varphi' )' - \ab( r^{N-1} \vab{\varphi'}^{q-2} \varphi' )'
	\\
	= \ &
	\e \ab( r^{N-2}  \ab( \log r )^{-2}  )' + \e^{q-1} \ab( r^{N-q} \ab( \log r )^{-2(q-1)} )'
	\\
	= \ & 
	(N-2) \e r^{N-3} \ab( \log r  )^{-2} -2 \e r^{N-3} \ab( \log r )^{-3} 
	\\
	& +(N-q) \e^{q-1} r^{N-q-1} \ab( \log r )^{-2(q-1)} 
	-2(q-1) \e^{q-1} r^{N-q-1} \ab( \log r)^{-2q +1}
	\\
	= \ & \e r^{N-3} \ab( \log r )^{-3} 
	\ab[ (N-2) \log r - 2 + \e^{q-2} (N-q) r^{2-q} \ab( \log r )^{ -2q + 5 } -2(q-1) \e^{q-2} r^{2-q} \ab( \log r )^{-2q + 4}  ]. 
\end{aligned}
\]
Since $2<q$, $N-2 \leq 0$ and $ N-q < 0$, 
\[
-(r^{N-1} \varphi')' - \ab( r^{N-1} \vab{\varphi'}^{q-2} \varphi'  )' < - 2 \e r^{N-3} (\log r)^{-3} < 0 \quad \text{in $(R_1,\infty)$}. 
\]
Choose $\e \in (0,1)$ so that $\varphi(R_1) < u(R_1)$, set $v \coloneq u - \varphi$ and 
write $A(s) \coloneq s + \vab{s}^{q-2} s \in C^1(\R)$. Then the above inequality and \eqref{eq_rad} lead to  
\[
- \ab( r^{N-1} A ( \varphi') )' < 0 < \alpha r^{N-1} u^{p-1} = - \ab( r^{N-1} A(u') )' \quad \text{in} \ (R_1,\infty).
\]
Recalling $u'<0$ and $\varphi'< 0$ in $(R_1,\infty)$, we see that 
$A'( \varphi' + \theta (u'-\varphi') ) \in C^\infty((R_1,\infty))$ for any $\theta \in [0,1]$ and that 
\begin{equation}\label{eq_comp}
	0 < - \Bab{ r^{N-1} \ab[ A(u') - A(\varphi') ] }'
	= - \Bab{ r^{N-1} \int_0^1 A' \ab( \varphi' + \theta (u'-\varphi') ) \odif{\theta} \, v'  }'.
\end{equation}
Since $v(R_1) > 0$ and $v(r) \to 0$ as $r \to \infty$, if $\inf_{(R_r,\infty)} v < 0$, 
then we may find $r_0 \in (R_1,\infty)$ such that $v(r_0) = \min_{(R_1,\infty)} v < 0$. 
From $v'(r_0) = 0$, \eqref{eq_comp} with $r=r_0$ gives 
\[
0 < - r_0^{N-1} \int_0^1 A' \ab( \varphi' (r_0) + \theta (u' (r_0)-\varphi' (r_0) ) ) \odif{\theta} \, v''(r_0). 
\]
The fact $A' (s) =1 + (q-1) \vab{s}^{q-2} >0$ leads to $v''(r_0) < 0$, however, this contradicts $v(r_0) = \min_{(R_1,\infty)} v$. 
Thus, $\inf_{(R_1,\infty)} v \geq 0$ and $ u \geq \varphi$ in $(R_1,\infty)$. 
Therefore, we obtain the following contradiction: 
\[
\infty > \frac{\Vab{u}_2^2}{(2\pi)^{N-1}} \geq \int_{R_1}^\infty r^{N-1} u^2 (r) \odif{r} \geq \int_{R_1}^\infty \frac{\e^2 r^{N-1}}{\ab( \log r )^2} \odif{r} = \infty.
\]
From this argument, when $N=1,2$ and $q>2$, \eqref{eq_zero_mass} admits only the trivial solution. 

\smallskip

\noindent
\textbf{The case $N = 3,4$:} In this case, for $\e \in (0,1)$, we consider 
\[
\varphi(r) \coloneq \e r^{-(N-2)}. 
\]
Since 
\[
\varphi'(r) = - (N-2) \e r^{-(N-1)}, \quad -r^{N-1} \vab{\varphi'(r)}^{q-2} \varphi'(r) = (N-2)^{q-1} \e^{q-1} r^{- (N-1)(q-2) },
\]
it follows that 
\[
\begin{aligned}
	-\ab( r^{N-1} \varphi' )' - \ab( r^{N-1} \vab{\varphi'}^{q-2} \varphi' )' 
	= 
	- (N-2)^{q-1} \e^{q-1} (N-1) (q-2)  r^{ -(N-1) (q-2)  - 1} < 0.
\end{aligned}
\]
Arguing as in the case $N=2$, we can prove that $u \geq \varphi$ in $(R_1,\infty)$. 
When $N=3,4$, from 
\[
\int_{R_1}^\infty r^{N-1} \varphi^2 \odif{r} = \e^2 \int_{R_1}^\infty r^{ - N + 3 } \odif{r} = \infty,
\]
it follows that $\varphi, u \not \in L^2(\RN)$ and this is a contradiction.

(ii) 
We treat the case $N \geq 3$, $2 < q < \infty$ and $2 < p \leq 2^*$. 
Let $u \in X$ be a solution to \eqref{eq_rad}. 
Then in view of \cref{lem_Poho} and $\Vab{\nabla u}_2^2 + \Vab{\nabla u}_q^q = \alpha \Vab{u}_p^p$, we obtain 
\[
0 = \frac{N-2}{2N} \Vab{\nabla u}_2^2 + \frac{N-q}{qN} \Vab{\nabla u}_q^q - \frac{\alpha}{p} \Vab{u}_p^p 
= \ab( \frac{1}{2^*} - \frac{1}{p} ) \Vab{\nabla u}_2^2 + \ab( \frac{1}{q} - \frac{1}{N} - \frac{1}{p} ) \Vab{\nabla u}_q^q. 
\]
From $2< p \leq 2^*$ and $2<q < \infty$, it is easily seen that 
\[
\frac{1}{2^*} - \frac{1}{p} \leq 0, \quad \frac{1}{q} - \frac{1}{N} - \frac{1}{p} < 0.
\]
Therefore, $\Vab{\nabla u}_q = 0$ holds, which yields $u \equiv 0$ by $u \in X$. 
Notice that the above argument is valid when $N=1,2$, $2<p,q<\infty$ and we omit the details. 
\end{proof}

\section{Multiplicity of solutions in the zero mass case}\label{MountZero}

The aim of this section is to prove \cref{t:zeromass}. 
Suppose $N \geq 3$, $q>2$, $p \in (2^*,q^*)$ and $\alpha > 0$. 
Let us define 
\[
\begin{aligned}
	\wh{X} &\coloneq \Set{ u \in \cD^{1,2} (\RN) | \,\,|\nabla u| \in L^q(\RN) }, \quad 
	\Vab{u}_{\wh{X}} \coloneq \Vab{\nabla u }_{2} + \Vab{\nabla u}_q,\\
	\wh{X}_\rad &\coloneq \Set{ u \in \wh{X} | \text{ $u$ is radially symmetric}}.
\end{aligned}
\]
Since $\wh{X} \subset L^r(\RN)$ holds for $r \in [2^*,q^*)$, 
$E_\alpha \in C^1(\wh{X} , \R)$. 
To find solutions to \eqref{eq_zero_mass}, 
we show the existence of critical points of $E_\alpha$ on $\wh{X}_\rad$. 
Notice that there is no constraint on the $L^2$-norm of solutions and hence we consider $E_\alpha$ on the entire space $\wh{X}_\rad$.

To find infinitely many critical points of $E_\alpha$ on $\wh{X}_\rad$ under the conditions above, 
we shall use \cite[Theorem 1.7]{ByIkMaMa24}. 
Define 
\[
\Psi(u) \coloneq \frac{1}{2} \Vab{\nabla u}_2^2 + \frac{1}{q} \Vab{\nabla u}_q^q \in C^1( \wh{X}_\rad , \R ), \quad 
\Phi(u) \coloneq  \frac{\alpha}{p} \Vab{u}_p^p \in C^1( \wh{X}_\rad , \R ).
\]
For $\lambda \in [1-\e,1]$ with $0 < \e \ll 1$, set 
\[
E_{\lambda,\alpha} (u) \coloneq \lambda \Psi(u) - \Phi(u) \in C^1( \wh{X}_\rad , \R ). 
\]
Of course $E_{1,\alpha}=E_\alpha$. 
To exploit \cite[Theorem 1.7]{ByIkMaMa24}, we consider the following conditions on $\Psi$ and $\Phi$: 

\begin{enumerate}[label=($\Psi$\arabic*),ref=$\Psi$\arabic*]
	\item \label{Psi1}
	$\Psi(0) = 0$ and 
	$\Psi$ is lower semicontinuous and convex;
	
	\item \label{Psi2}
	$\Psi(u) \to \infty$ as $\Vab{u}_{\wh{X}} \to \infty$;
	
	\item \label{Psi3}
	If $u_j \rightharpoonup u_\infty$ weakly in $\wh{X}_\rad$ and $\Psi(u_j) \to \Psi(u_\infty)$, then 
	$\Vab{u_j - u_\infty}_{\wh{X}} \to 0$;
\end{enumerate}

\begin{enumerate}[label=($\Phi$\arabic*),ref=$\Phi$\arabic*]
	\item \label{Phi1}
	$\Phi \in C^1(\wh{X}_\rad  , \R  )$;
	
	\item \label{Phi2}
	If $u_n \rightharpoonup u_\infty$ weakly in $\wh{X}_\rad$, then 
	\[
	\liminf_{n \to \infty} \Phi'(u_n) (v-u_n) \geq \Phi'(u_\infty) (v-u_\infty) \quad \text{for any $v \in \wh{X}_\rad$};
	\]
\end{enumerate}

\begin{enumerate}[label=(IB),ref=IB]
	\item \label{IB}
	For any $0 < a \leq b < \infty$ and $c \in \R$, the set 
	\[
	\mathcal{K}^c_{[a,b]} \coloneq \Set{ u \in \wh{X}_\rad | \text{for some $\lambda \in [a,b]$, $E_{\lambda,\alpha} '(u) = 0$ and $E_{\lambda,\alpha}(u) \leq c$} }
	\]
	is bounded in $\wh{X}_\rad$;
\end{enumerate}

\begin{enumerate}[label=(E),ref=E]
	\item \label{E}
	$\Psi$ and $\Phi$ are even; 
\end{enumerate}

\begin{enumerate}[label=(SMP),ref=SMP]
	\item \label{SMP}
	there exist $\rho_0 > 0$, $\overline{\e} > 0$ and odd maps $\pi_{0,k} \in C(\partial D^k, \wh{X}_\rad)$ for each $k \in \N$ 
	where $D^k \coloneq \Set{ \sigma \in \R^k | \vab{\sigma} < 1 }$ such that 
	\[
	\begin{aligned}
		&\inf \Set{ E_{1-\overline{\e} , \alpha } (u) | u \in \wh{X}_\rad, \ \Vab{u}_{\wh{X}} = \rho_0   } > 0, 
		\\
		&\min_{\sigma \in \partial D^k } \Vab{\pi_{0,k} (\sigma)}_{\wh{X}} > \rho_0, \quad 
		\sup_{ \sigma \in \partial D^k } E_{1+\overline{\e},\alpha} ( \pi_{0,k} (\sigma) ) < 0.
	\end{aligned}
	\]
\end{enumerate}

Then the following result is obtained in \cite[Theorem 1.7]{ByIkMaMa24}: 

\begin{theorem}\label{theoMountzero}
Assume \eqref{Psi1}, \eqref{Psi2}, \eqref{Psi3}, \eqref{Phi1}, \eqref{Phi2}, \eqref{IB}, \eqref{E} and \eqref{SMP}. 
Then $E_{1,\alpha}$ admits infinitely many critical points $(u_k)_k\subset \wh{X}_\rad$ with $E_{1,\alpha}(u_k)\to\infty$ as $k\to\infty$.
\end{theorem}

\begin{proof}[Proof of \cref{t:zeromass} (the existence of solutions)]
By \cref{theoMountzero}, it suffices to verify \eqref{Psi1}, \eqref{Psi2}, \eqref{Psi3}, \eqref{Phi1}, \eqref{Phi2}, \eqref{IB}, \eqref{E} and \eqref{SMP}. 
With regard to \eqref{Psi1}, \eqref{Psi2}, \eqref{Phi1} and \eqref{E}, these are easy to check. 
In what follows, we check \eqref{Psi3}, \eqref{Phi2}, \eqref{IB} and \eqref{SMP}.

\noindent
\textbf{Proof of \eqref{Psi3}:} 
Let $u_j \rightharpoonup u_\infty$ weakly in $\wh{X}_\rad$ and $\Psi(u_j) \to \Psi(u_\infty)$. 
The weak lower semicontinuity of norms yields $\Vab{\nabla u_\infty}_2^2 \leq \liminf_{n \to \infty} \Vab{\nabla u_n}_2^2$ 
and $\Vab{\nabla u_\infty}_q^q \leq \liminf_{n \to \infty} \Vab{\nabla u_n}_q^q$. 
From the uniform convexity of $L^2$ and $L^q$, it suffices to show 
\begin{equation}\label{eq:conv-norm}
	\Vab{\nabla u_n}_2^2 \to \Vab{\nabla u_\infty}_2^2, \quad \Vab{\nabla u_n}_q^q \to \Vab{\nabla u_\infty}_q^q. 
\end{equation}
By 
\[
\Psi(u_\infty) = \frac{1}{2} \Vab{\nabla u_\infty}_2^2 + \frac{1}{q} \Vab{\nabla u_\infty}_q^q 
\leq \frac{1}{2} \liminf_{n \to \infty} \Vab{\nabla u_n}_2^2 + \frac{1}{q} \liminf_{n \to \infty} \Vab{\nabla u_n}_q^q 
\leq \liminf_{n \to \infty} \Psi(u_n) = \Psi(u_\infty),
\]
we deduce that 
\[
\Vab{\nabla u_\infty}_2^2 = \liminf_{n \to \infty} \Vab{\nabla u_n}_2^2, \quad \Vab{\nabla u_\infty}_q^q = \liminf_{n \to \infty} \Vab{\nabla u_n}_q^q. 
\]
If $\Vab{\nabla u_\infty}_2^2 = \liminf_{n \to \infty} \Vab{\nabla u_n}_2^2 < \limsup_{n \to \infty} \Vab{\nabla u_n}_2^2$, 
then by choosing a subsequence $(u_{n_k})_{k \in \N}$ so that 
$\Vab{\nabla u_{n_k}}_2^2 \to \limsup_{n \to \infty} \Vab{\nabla u_n}_2^2$, we get the following contradiction: 
\[
\Psi(u_\infty) < \frac{1}{2} \lim_{k \to \infty} \Vab{\nabla u_{n_k}}_2^2 + \frac{1}{q} \liminf_{k \to \infty} \Vab{\nabla u_{n_k}}_q^q 
\leq \limsup_{k \to \infty} \Psi(u_{n_k}) = \Psi(u_\infty).
\]
Hence, $\Vab{\nabla u_n}_2^2 \to \Vab{\nabla u_\infty}_2^2$. Similarly, $\Vab{\nabla u_n}_q^q \to \Vab{\nabla u_\infty}_q^q$ may be obtained 
and \eqref{eq:conv-norm} holds.

\noindent
\textbf{Proof of \eqref{Phi2}:}
Let $u_n \rightharpoonup u_\infty$ weakly in $\wh{X}_\rad$. According to \cite{Li82}, 
$u_n \to u_\infty$ strongly in $L^r(\RN)$ for any $r \in (2^*,q^*)$. 
Therefore, 
\[
\begin{aligned}
	\Phi'(u_n) (v-u_n) &= \alpha \int_{\RN} \vab{u_n}^{p-2} u_n (v - u_n) \odif{x}   
	\\
	&\to \alpha \int_{\RN} \vab{u_\infty}^{p-2} u_\infty (v- u_\infty) \odif{x} = \Phi'(u_\infty) (v-u_\infty),
\end{aligned}
\]
for any $v\in \wh{X}_\rad$ which completes the proof.

\noindent
\textbf{Proof of \eqref{IB}:} 
Let $u \in \mathcal{K}^c_{[a,b]}$. Since $u \in \wh{X}_\rad$ is a weak radial solution to 
\[
\lambda \ab( -\Delta u - \Delta_q u ) = \alpha \vab{u}^{p-2} u \quad \text{in} \ \RN,
\]
the Pohozaev identity (\cref{lem_Poho}) holds: 
\[
0= \lambda \frac{N-2}{2N} \Vab{\nabla u}_2^2 + \lambda \frac{N-q}{Nq} \Vab{\nabla u}_q^q - \frac{\alpha}{p} \Vab{u}_p^p 
= E_{\lambda,\alpha} (u) - \frac{\lambda}{N} K(u) .
\]
From $\lambda \in [a,b]$ and 
\[
c \geq E_{\lambda,\alpha}(u) 
= \frac{\lambda}{N} K(u), 
\]
$\mathcal{K}^c_{[a,.b]}$ is bounded in $ \wh{X}_\rad$. 

\noindent
\textbf{Proof of \eqref{SMP}:} We divide the proof into points.

\noindent
\textbf{The existence of $\rho_0$ when $q<N$:}
By $2<q<N$ and $2^*<p < q^*$, since 
\[
\vab{s}^p \leq \vab{s}^{2^*} + \vab{s}^{q^*} \quad \text{for all $s \in \R$}, 
\]
the fact $2<q<N$ and Sobolev's inequality yield 
\begin{equation*}
	\Vab{u}_p^p \leq \Vab{u}_{2^*}^{2^*} + \Vab{u}_{q^*}^{q^*} \leq C \Vab{\nabla u}_2^{2^*} + C \Vab{\nabla u}_q^{q^*}.
\end{equation*}
Hence, if $\Vab{u}_{\wh{X}} = \rho_0$, then 
\[
\begin{aligned}
	E_{1-\bar{\e}, \alpha}(u) 
	&\geq 
	\frac{1}{2} \ab[ (1-\e) - C \Vab{\nabla u}_2^{2^*-2} ] \Vab{\nabla u}_2^2 + \frac{1}{q} \ab[ (1-\e) - C \Vab{\nabla u}_q^{q^*-q} ] \Vab{\nabla u}_q^q 
	\\
	&\geq 
	\frac{1}{2} \ab[ (1-\e) - C \rho_0^{ (2^*-2)/2 } ] \Vab{\nabla u}_2^2 + \frac{1}{q} \ab[ (1-\e) - C \rho_0^{ (q^*-q)/q } ] \Vab{\nabla u}_q^q.
\end{aligned}
\]
Since $\bar{\e}>0$ can be small, if $\rho_0 > 0$ is sufficiently small, then it is easily seen that 
\[
\inf_{ \Vab{u}_{\wh{X}} = \rho_0 } E_{1-\bar{\e},\alpha}(u) 
\geq \inf_{ \Vab{u}_{\wh{X}} = \rho_0 } \frac{1}{2q} \ab( \Vab{\nabla u}_2^2 + \Vab{\nabla u}_q^q ) > 0. 
\]

\noindent
\textbf{The existence of $\rho_0$ when $q \geq N$:}
Fix $r \in (2,N)$ so that $\max \set{p,q} < r^* = Nr/(N-r)$. From 
\[
\vab{s}^p \leq \vab{s}^{2^*} + \vab{s}^{r^*} \quad \text{for each $s \in \R$}, 
\]
it follows that 
\[
\Vab{u}_p^p \leq \Vab{u}_{2^*}^{2^*} + \Vab{u}_{r^*}^{r^*}  \leq C \Vab{\nabla u}_2^{2^*} + C \Vab{\nabla u}_r^{r^*}. 
\]
Choose $\theta \in (0,1)$ with $1/r = \theta/2 + (1-\theta)/q$. The interpolation and Young's inequality give 
\[
\Vab{\nabla u}_r \leq \Vab{\nabla u}_2^{\theta} \Vab{\nabla u}_q^{1-\theta} 
\leq \theta \Vab{\nabla u}_2 + (1-\theta) \Vab{\nabla u}_q. 
\]
Hence, 
\[
\Vab{u}_p^p \leq C \Vab{\nabla u}_2^{2^*} + C \ab( \Vab{\nabla u}_2^{r^*} + \Vab{\nabla u}_q^{r^*} ). 
\]
By $2<2^*$ and $2<q<r^*$, the rest of the argument is similar to the proof in the above.

\noindent
\textbf{The existence of $\pi_{0,k}$:}
Let $k \in \N$ and $\varphi_{j_1},\dots, \varphi_{j_k}\in C^\infty_{c,\rad}( \RN )$ satisfy 
\[
\supp \varphi_{j_1} \cap \supp \varphi_{j_2} = \emptyset \ \text{if $j_1 \neq j_2$}, \quad \Psi(\varphi_j) = 1 \quad \text{for each $j=1,\dots,k$}.
\]
For $\sigma \in \partial D^k$, set 
\[
\wt{\pi}_{k} (\sigma) \coloneq \sum_{j=1}^k \sigma_j \varphi_j \in C(\partial D^k, \wh{X}_\rad).
\]
It is easily seen that $\wt{\pi}_k$ is odd and 
\[
\beta_k \coloneq \min_{\sigma \in \partial D^k } \Vab{ \wt{\pi}_{k} (\sigma)  }_p^p > 0, 
\quad \gamma_k \coloneq \max_{\sigma \in \partial D^k} \ab(  \Vab{ \nabla  \wt{\pi}_k(\sigma) }_2^2 + \Vab{ \nabla  \wt{\pi}_k(\sigma) }_q^q )  < \infty. 
\]
For $R \gg 1$, consider 
\[
\pi_{0,k} (\sigma) (x) \coloneq \wt{\pi}_k (\sigma) \ab( \frac{x}{R} ). 
\]
Then $\pi_{0,k} \in C(\partial D^k, \wh{X}_\rad)$ is an odd map and satisfies 
\[
\begin{aligned}
	E_{1+\overline{\e}, \alpha } ( \pi_{0,k} (\sigma) ) 
	&= (1+\overline{\e}) \Bab{ \frac{R^{N-2}}{2} \Vab{\nabla  \wt{\pi}_k (\sigma) }_2^2 + \frac{R^{N-q}}{q} \Vab{ \nabla  \wt{\pi}_k(\sigma) }_q^q  } 
	- \frac{\alpha}{p} R^N \Vab{  \wt{\pi}_k(\sigma) }_p^p
	\\
	&\leq 
	\ab( 1 + \overline{\e} ) \ab( R^{N-2} + R^{N-q} ) \gamma_k- \frac{\alpha}{p} \beta_k R^N.
\end{aligned}
\]
Since $N \geq 3$ and 
\[
\min_{\sigma \in \partial D^k } \Vab{\nabla \pi_{0,k} (\sigma)}_2^2 \to \infty \quad \text{as $R \to \infty$},
\]
for a sufficiently large $R\gg1$, we obtain 
\[
\max_{\sigma \in \partial D^k}  E_{1+\overline{\e}} ( \pi_{0,k} (\sigma) ) < 0, \quad 
\min_{\sigma \in \partial D^k } \Vab{\pi_{0,k} (\sigma)}_{\wh{X}} > \rho_0.
\]
Hence \eqref{SMP} holds. 
\end{proof}

To complete the proof of \cref{t:zeromass}, we shall derive the following decay estimates of 
solutions to \eqref{eq_zero_mass} in $\wh{X}_\rad$:

\begin{proposition}\label{prop_decay}
	Assume $N \geq 3$, $2 < q < \infty$, $2^* < p < q^*$ and $u \in \wh{X}_{rad}$ is a solution to \eqref{eq_zero_mass}.
	Then there exists $C>0$ such that $\vab{u(r)} \leq C r^{-(N-2)}$ for all $r \geq 1$. 
	In particular, when $N \geq 5$, $u \in L^2(\RN)$. 
    \end{proposition}

\begin{remark}\label{Rem:sharpness}
We remark that to obtain the decay estimate $\vab{u(r)} \leq Cr^{-N+2}$ as in \cref{prop_decay}, 
the condition $p>2^*$ is sharp when $N \geq 5$. 
In fact, let $p=2^*$, $N \geq 5$ and $u \in \wh{X}_\rad \setminus \set{0}$ be a critical point of $E_\alpha$. 
Then \cref{nonex1} (ii) asserts $u \not \in X$ and the estimate $\vab{u(r)} \leq C r^{-N+2}$ cannot be obtained; 
otherwise $u \in L^2(\RN)$ and $u \in X$ hold. 
\end{remark}

\begin{proof}[Proof of \cref{prop_decay}] 
	We divide the proof into intermediate claims.
	
	\smallskip 
	
	\noindent
	\textbf{Claim 1:} \textsl{The space $C^\infty_{c,\rad} (\RN)$ is dense in $\wh{X}_\rad$.}
	
	\smallskip 
	
	Choose $u \in \wh{X}_\rad$ arbitrarily and let $\zeta_1 \in C^\infty_{c,\rad}(\RN)$ satisfy 
	\[
	0 \leq \zeta_1 \leq 1, \quad \zeta_1 = 1 \quad \text{in} \ B_1, \quad \zeta_1 = 0 \quad \text{in} \ \RN \setminus B_2.
	\]
	Set $\zeta_n(x) \coloneq \zeta_1(x/n)$ and $u_n \coloneq \zeta_n u$. 
	Since each $u_n$ has compact support, it can be approximated by functions in $C^\infty_{c,\rad}(\RN)$. 
	Hence, it suffices to show $\Vab{\nabla u_n - \nabla u}_{L^2\cap L^q} \to 0$. 
	To this end, from $\nabla u_n = (\nabla \zeta_n) u + \zeta_n \nabla u$ and the dominated convergence theorem, 
	it suffices to show 
	\begin{equation}\label{conv-rem}
		\int_{\RN} \vab{\nabla \zeta_n}^2 \vab{u}^2 + \vab{\nabla \zeta_n}^q \vab{u}^q \odif{x} \to 0.
	\end{equation}
	When $2 \leq q \leq 2^*$ (the case $q=2$ is included), 
	\[
	\int_{\RN}\vab{\nabla \zeta_n}^q \vab{u}^q \odif{x} \leq \frac{C}{n^q} \int_{n\leq \vab{x} \leq 2n } \vab{u}^q \odif{x} 
	\leq \frac{C}{n^q} \ab( \int_{ n \leq \vab{x} \leq 2n } \vab{u}^{2^\ast} \odif{x}  )^{q/2^\ast} 
	\vab{ B_{2n} }^{ 1 - q/2^\ast }.
	\]
	Notice that $\vab{B_{2n}}^{1-q/2^*} \leq C n^{ N - Nq/2^\ast }$ and 
	$N - Nq/2^* - q \leq 0$ due to $q \geq 2$. 
	From
	\[
	\ab( \int_{ n \leq \vab{x} \leq 2n } \vab{u}^{2^\ast} \odif{x}  )^{q/2^\ast}  \to 0,
	\]
	when $2 \leq q \leq 2^*$, \eqref{conv-rem} holds.

	On the other hand, when $2^* < q$, choose $s \in (2,q)$ so that $s^* = q$ where $1/s^* \coloneq 1/s - 1/N$. 
	Sobolev's inequality with the interpolation inequality ($\Vab{\nabla u}_{s} \leq \Vab{\nabla u}_2^\theta \Vab{\nabla u}_q^{1-\theta} \leq \Vab{u}_{\wh{X}} $)
	leads to 
	\[
	\int_{\RN} \vab{\nabla \zeta_n}^q \vab{u}^q \odif{x} 
	\leq C n^{-q} \int_{\RN} \vab{u}^{s^*} \odif{x} 
	\leq C n^{-q}  \Vab{\nabla u}_s^{s^*} \leq C n^{-q} \Vab{u}_{\wh{X}}^{s^*} \to 0.
	\]
	Hence \eqref{conv-rem} holds in this case, too. Thus we complete the proof. 
	
\smallskip 
	
	\noindent
	\textbf{Claim 2:} \textsl{There exists $C>0$ such that for all $u \in \wh{X}_\rad$ and $r > 0$, 
		\begin{equation}\label{eq_decay}
			\vab{u(r)} \leq C r^{ - (N-2)/2  } \Vab{u}_{\wh{X}}.
		\end{equation}
	}

\smallskip 
	
	It suffices to prove \eqref{eq_decay} for $u \in C^\infty_{c,\rad} (\RN)$. 
	Set $\beta \coloneq (2^*+2)/2 = 2(N-1)/(N-2) $. Then for any $r \in (0,\infty)$, 
	\[
	\begin{aligned}
		\vab{u(r)}^\beta = \int_r^\infty - \odv*{\vab{u(s)}^\beta }{s} \odif{s} 
		&\leq 
		C \int_r^\infty \vab{u(s)}^{\beta -1} \vab{u'(s)} \odif{s} 
		\\
		&=\int_r^\infty \Bab{\vab{u(s)}^{\beta -1} s^{-(N-1)/2}} \Bab{ s^{(N-1)/2}  \vab{u'(s)} } \odif{s}
		\\
		&	\leq 
		C \ab( \int_r^\infty s^{-2(N-1)}  s^{N-1} \vab{u(s)}^{2(\beta -1)} \odif{s}  )^{1/2} \Vab{\nabla u }_{L^2(\RN \setminus B_r)}
		\\
		& \leq C r^{-(N-1)} \Vab{u}_{2^*}^{ N/(N-2) } \Vab{\nabla u}_{L^2 (\RN \setminus B_r)} 
		\\
		& \leq C r^{-(N-1)} \Vab{\nabla u}_2^{ (2N-2)/(N-2) }
	\end{aligned}
	\]
	where we also used H\"older's inequality.
	Hence, 
	\[
	\vab{u(r)} \leq C r^{ - (N-2)/2 } \Vab{\nabla u}_2 \leq C r^{-(N-2)/2} \Vab{u}_{\wh{X}},
	\]
	which completes the proof.

\smallskip 	
	
	\noindent
	\textbf{Claim 3:} 
	\textsl{$\vab{u'(r)} \to 0$ as $r \to \infty$.}

\smallskip

	Set 
	\[
	F(r) \coloneq \frac{1}{2} \vab{u'(r)}^2 +  \frac{q-1}{q} \vab{u'(r)}^q + \frac{\alpha}{p} \vab{u(r)}^p.
	\]
	From \eqref{eq_rad-ver2}, we compute
	\[
	\begin{aligned}
		F'(r) 
		&= u'(r) u''(r) + (q-1) \vab{u'(r)}^{q-2} u'(r) u''(r) + \alpha \vab{u(r)}^{p-2} u(r) u'(r)
		\\
		&= - \frac{N-1}{r} \ab( 1 + \vab{u'(r)}^{q-2} ) \ab( u'(r) )^2 \leq 0.
	\end{aligned}
	\]
	Since $F(r) \geq 0$ for any $r \in [0,\infty)$, the limit $\lim_{r \to \infty} F(r) \in [0, \infty)$ exists. 
	By $u(r) \to 0$ as $r \to \infty$, we infer that $u'(r) \to 0$ and $F(r) \to 0$ as $r \to \infty$ (since $u \in C^1$).

\smallskip

	\noindent
	\textbf{Claim 4:} 
	\textsl{There exists $C=C(u)$ such that $\vab{u(r)} \leq C r^{-(N-2)}$ for all $r \geq 1$. }

\smallskip

	We first consider the case $p > 2^* + 1$. Notice that 
	\[
	N-1 - (p-1) \frac{N-2}{2} < N -1 - N = - 1. 
	\]
	From \textbf{Claim 2} and $u \in \wh{X}_\rad$, it follows that 
	\[
	r^{N-1} \vab{u(r)}^{p-1} \leq C r^{ N-1 - (p-1) (N-2)/2 }  \in L^1((1,\infty)). 
	\]
	Thus, by \eqref{eq_rad}, as $r \to \infty$, 
	$(r^{N-1} (1+ \vab{u'(r)}^{q-2} ) u'(r)  )$ forms a Cauchy sequence and hence the limit 
	\[
	\lim_{r \to \infty} r^{N-1} \ab( 1 + \vab{u'(r)}^{q-2} ) u'(r) = C \in \R
	\]
	exists. Since $u'(r) \to 0$ as $r \to \infty$, there exists $C'>0$ such that
	\[
	\vab{u'(r)} \leq C r^{-(N-1)} \quad \text{for all $r \geq 1$}. 
	\]
	From $u(r) \to 0$ as $r \to \infty$, we obtain 
	\[
	\vab{u(r)} \leq \int_r^\infty \vab{u'(s)} \odif{s} \leq C r^{-(N-2)} \quad \text{for each $r \geq 1$}. 
	\]
	This is the desired conclusion.

	We next treat the case $2^* < p \leq 2^* + 1$. 
	By integrating \eqref{eq_rad} on $[1,r]$ and noting \textbf{Claim 3}, there exists $C>0$ such that for each $r>1$, 
	\begin{equation}\label{eq_basic}
		r^{N-1} \vab{u'(r)} \leq C \ab[ 1+ \int_1^r s^{N-1} \vab{u(s)}^{p-1} \odif{s}].
	\end{equation}
	Write $a_0 \coloneq (N-2)/2 $. \textbf{Claim 2} and \eqref{eq_basic} yield 
	\begin{equation}\label{eq_basic2}
		\vab{u'(r)} \leq C r^{-(N-1)} \ab( 1 + \int_1^r s^{ N-1 - a_0(p-1) } \odif{s} ). 
	\end{equation}
	When $N - a_0(p-1) = 0$, that is $p = 2^*+1$, then 
	\[
	\vab{u'(r)} \leq C' r^{-(N-1)} \ab( 1 + \log r ) \in L^1((1,\infty))
	\]
	and 
	\[
	\begin{aligned}
		\vab{u(r)} \leq \int_r^\infty \vab{u'(s)} \odif{s} 
		&\leq 
		C' \int_r^\infty s^{1-N} \odif{s} + C' \int_r^\infty s^{1-N} \log s \odif{s}
		\\
		&\leq C' r^{2-N} + C' \int_{\log r}^\infty e^{\tau(1-N)} \tau e^\tau \odif{\tau}
		\\
		&\leq C' r^{2-N} + C' \ab[ - \frac{e^{\tau(2-N)}}{N-2} \tau ]^\infty_{\log r}  +C'' \int_{\log r}^{\infty} e^{(2-N) \tau} \odif{\tau}
		\\
		& \leq C''' r^{2-N} \ab( 1 + \log r ). 
	\end{aligned}
	\]
	It follows from $p > 2^*$ that 
	\[
	- (p-1) (N-2) < - (2^*-1) (N-2) = - (N+2). 
	\]
	Substituting the estimate $\vab{u(r)} \leq C r^{2-N} (1+ \log r)$ into \eqref{eq_basic} gives 
	\[\begin{aligned}
		\vab{u'(r)} 
		&\leq Cr^{1-N}  \ab[ 1 + \int_1^r s^{N-1} s^{ - (N+2) }\cdot s^{- (p-1) (N-2)+(N+2)}( 1 + \log s )^{p-1} \odif{s} ] \\
		&\leq Cr^{1-N} \ab[ 1 + \int_1^r s^{N-1} s^{ - (N+2) } \odif{s} ] \leq C r^{1-N}. 
	\end{aligned}\]
			Therefore, 
			\[
			\vab{u(r)} \leq \int_r^\infty \vab{u'(s)} \odif{s} \leq C r^{2-N}
			\]
			and the desired decay estimate holds.

			When $2^* < p < 2^* + 1$, 
			\[
			N-1 - a_0(p-1) > N - 1 - \frac{N-2}{2} \cdot 2^* =  -1
			\]
			and 
			\[
			1- a_0(p-1) < 1 - \frac{N-2}{2} \cdot \frac{N+2}{N-2} = - \frac{N}{2} < -1. 
			\]
			Hence, \eqref{eq_basic2} implies 
			\[
			\vab{u'(r)} \leq C' r^{-(N-1)} \ab(  1 + r^{N - a_0(p-1) } ) = C' \ab( r^{-(N-1)} + r^{ 1- a_0(p-1) }  ) \in L^1((1,\infty)). 
			\]
			and 
			\begin{equation}\label{step0}
			\vab{u(r)} \leq \int_r^\infty \vab{u'(s)} \odif{s} \leq C \ab( r^{-(N-2)} + r^{2 - a_0(p-1)} ) . 
			\end{equation}
			Notice that 
			\[
			2 - a_0(p-1) < - a_0 \ \iff \ 2 < a_0 (p-2) \ \iff \ 2^* = \frac{4}{N-2} + 2 < p. 
			\]
			Set 
			\[
			a_1 \coloneq (p-1) a_0 - 2 > a_0. 
			\]
			If $a_1 \geq N-2$, then we get the desired conclusion due to \eqref{step0}. 
			Assume $a_1 < N-2$. Substituting $\vab{u(r)} \leq C r^{ - a_1 }$ into \eqref{eq_basic} implies 
			\[
			\vab{u'(r)} \leq C r^{1-N} \ab[ 1 + \int_1^r s^{N-1 - (p-1) a_1 } \odif{s} ]. 
			\]
			When $N \leq (p-1) a_1$, as arguing in the above, we have 
			\[
			\vab{u(r)} \leq C r^{2-N} \ab( 1 + \log r ),
			\]
			and substitution into \eqref{eq_basic} again gives $\vab{u(r)} \leq C r^{-(N-2)}$. 
			
			If $N > (p-1) a_1$, then 
			\[
			\vab{u'(r)} \leq C r^{1-N} \ab[ 1 + r^{ N - (p-1)a_1  } ] = C \ab[ r^{1-N} + r^{1 - (p-1)a_1} ].
			\]
			Since
			\[
			a_1 > a_0, \quad 1 - a_1(p-1) < 1 - a_0(p-1) < - \frac{N}{2},
			\]
			we have 
			\[
			\vab{u(r)} \leq \int_r^\infty \vab{u'(s)} \odif{s} \leq C \ab[ r^{2-N} + r^{ 2 - (p-1) a_1 } ]. 
			\]
			Set 
			\[
			a_2 \coloneq (p-1) a_1 - 2. 
			\]
			From $a_0 < a_1$ it follows that
			\[
			a_2 - a_1 = (p-1) (a_1 - a_0) > 0.
			\]

			The next step is to substitute $\vab{u(r)} \leq C r^{-a_2}$ (here assuming $a_2 < N-2$,  otherwise the claim is proved) into \eqref{eq_basic} to get 
			\[
			\vab{u'(r)} \leq C r^{1-N} \ab( 1 +  \int_1^r s^{N-1-(p-1)a_2} \odif{s} ). 
			\]
			When $N \leq (p-1)a_2$, then the desired estimate holds from the above argument. 
			If $N > (p-1)a_2$, then 
			\[
			\vab{u'(r)} \leq C \ab[ r^{1-N} + r^{ 1 - (p-1)a_2 } ]
			\]
			and 
			\[
			\vab{u(r)} \leq C \ab[ r^{2-N} + r^{2 - (p-1) a_2} ].
			\]
			Thus, set 
			\[
			a_3 \coloneq (p-1) a_2 - 2
			\]
			and notice that 
			$a_3 - a_2 = (p-1)(a_2-a_1) = (p-1)^2 (a_1-a_0) > 0$.

			In general, set 
			\[
			a_{n+1} \coloneq (p-1) a_n - 2.
			\]
			Then it can be shown that if $(p-1) a_n \geq N$, then the desired result holds. Otherwise, 
			$\vab{u(r)} \leq C r^{-a_{n+1}}$ and $a_{n+1} - a_n = (p-1) (a_n-a_{n-1}) = (p-1)^n (a_1-a_0)$. 
			Since $p>2$, after finitely many times of iterations, 
			we observe that $a_n > N-2$ and obtain $\vab{u(r)} \leq C r^{2-N}$. 
			Hence we complete the proof of \textbf{Claim 4}, giving the complete proof. 
		\end{proof}

\begin{appendices}
\appendix

\section{Regularity and the Pohozaev identity}\label{App}



\begin{lemma}\label{l:reg}
	Let $N \geq 1$, $2 < q < \infty$, $p < q^*$, $\lambda \in \R$, $\alpha > 0$ and 
	$u \in X$ be a weak solution of \eqref{eq_main}. 
	Then $u \in C^{2,\gamma}_{\rm loc} (\RN)$ for some $\gamma \in (0,1)$. 
\end{lemma}

\begin{proof}
We first claim $u \in L^\infty(\RN)$. Indeed, when $q > N$, the claim follows from the embedding $X \subset L^\infty(\RN)$. 
On the other hand, when $2 < q \leq N$, we may exploit \cite[Theorems 1 and 2]{Se64} to show $u \in L^\infty(B_1(z))$ for each $z \in \RN$. 
Here we point out that this estimate can be uniform with respect to $z \in \RN$, and hence $u \in L^\infty(\RN)$ holds.

Since $u \in L^\infty(\RN)$, \cite[Theorem 1.7]{Li91} yields $u \in C^1(\RN)$. 
We introduce $A(\xi) \coloneq \xi + \vab{\xi}^{q-2} \xi \in C^1(\RN, \RN)$. 
Then $u \in X \cap L^\infty(\RN) \cap C^1(\RN)$ is a weak solution of 
\begin{equation}\label{eq:u}
	-\diver  A(\nabla v)   = f_u(x) \quad \text{in} \ \RN,
\end{equation}
where $f_u \coloneq - \lambda u + \alpha \vab{u}^{p-2} u \in C^1(\R^N)$. 
For $ 0 < \vab{h} \ll 1$, we test $-\Delta_{-he_i} \ab( \varphi^2 \Delta_{he_i} u ) $ to \eqref{eq:u} where 
$\varphi \in C^\infty_c(B_R(0))$ and $\Delta_{he_i} v(x) \coloneq ( v(x+he_i) - v(x) )/h$. Then 
\[
\int_{\RN} \Delta_{he_i}\ab( A ( \nabla u ) ) \nabla \ab(  \varphi^2 \Delta_{he_i} u ) \odif{x} 
= \int_{\RN} \ab( \Delta_{he_i} f_u ) \varphi^2 \Delta_{he_i} u \odif{x}. 
\]
From $A \in C^1(\RN,\RN)$ it follows that 
\[
\begin{aligned}
	\Delta_{he_i} \ab( A (\nabla u) ) (x)
	&= 
	h^{-1} \Bab{ A \ab( \nabla u(x) + \ab( \nabla u(x+he_i) - \nabla u(x) ) ) - A(\nabla u(x)) }
	\\
	&= 
	\int_0^1 DA \ab( \nabla u(x) + \theta \ab( \nabla u(x+he_i) - \nabla u(x) ) ) \odif{\theta} 
	\cdot \Delta_{he_i} \nabla u(x).
\end{aligned}
\]
Thus, by writing 
\[
B_{he_i}(x) \coloneq \int_0^1 DA \ab( \nabla u(x) + \theta \ab( \nabla u(x+he_i) - \nabla u(x) ) ) \odif{\theta},
\]
we obtain 
\[
\begin{aligned}
	&\int_{\RN} \Delta_{he_i} \ab( A(\nabla u) ) \nabla \ab( \varphi^2 \Delta_{he_i} u ) \odif{x}
	\\ 
	= \ &
	\int_{\RN} \ab[  B_{he_i} \Delta_{he_i} \nabla u \cdot \Delta_{he_i} \nabla u] \varphi^2  \odif{x} 
	+ 2 \int_{\RN} \ab[ B_{he_i} \Delta_{he_i} \nabla u \cdot \nabla \varphi ] \varphi \Delta_{he_i} u \odif{x}.
\end{aligned}
\]
A direct computation yields 
\begin{equation}\label{deri-A}
	DA(\xi)_{ij} = \ab( 1 + \vab{\xi}^{q-2} ) \delta_{ij} + (q-2) \vab{\xi}^{q-2} \frac{\xi_i}{\vab{\xi}} \frac{\xi_j}{\vab{\xi}}, \quad 
	DA(\xi) \geq \mathrm{Id}.
\end{equation}
Thus, 
\[
\int_{\RN} \vab{ \Delta_{he_i} \nabla u }^2 \varphi^2 \odif{x} + 2 \int_{\RN} \ab[ B_{he_i} \Delta_{he_i} \nabla u \cdot \nabla \varphi ] \varphi \Delta_{he_i} u \odif{x}
\leq 
\int_{\RN} \ab( \Delta_{he_i} f_u ) \varphi^2 \Delta_{he_i} u  \odif{x} .
\]

On the other hand, by $u \in C^1(\RN)$ and $\varphi \in C^\infty_c(B_R(0))$, one can find $C=C( \Vab{\nabla u}_{L^\infty(B_R)} )$ such that
\[
\begin{aligned}
	\vab{2 \int_{\RN} \ab[ B_{he_i} \Delta_{he_i} \nabla u \cdot \nabla \varphi ] \varphi \Delta_{he_i} u \odif{x}}
	&\leq
	C \int_{\RN}  \vab{\Delta_{he_i} \nabla u} \vab{\nabla \varphi} \vab{\varphi} \vab{\Delta_{he_i} u}  \odif{x}
	\\
	&\leq
	\frac{1}{2} \int_{\RN} \vab{ \Delta_{he_i} \nabla u }^2 \varphi^2 \odif{x} 
	+ C^2 \int_{\RN} \vab{\nabla \varphi}^2 \vab{\Delta_{he_i} u}^2 \odif{x}.  
\end{aligned}
\]
Therefore, we get 
\begin{equation}\label{eq:u2}
	\int_{\RN} \vab{ \Delta_{he_i} \nabla u }^2 \varphi^2  \odif{x}  
	\leq 
	2\int_{\RN} \ab( \Delta_{he_i} f_u ) \varphi^2 \Delta_{he_i} u \odif{x} + 2 C^2 \int_{\RN} \vab{\nabla \varphi}^2 \vab{ \Delta_{he_i} u }^2 \odif{x}.
\end{equation}
By $u \in C^1(\RN)$, $\Delta_{he_i} u \to \partial_i u$ in $C_{\rm loc} (\RN)$ as $h \to 0$, and hence 
the right-hand side of \eqref{eq:u2} is bounded. Since $\varphi \in C^\infty_c(B_R(0))$ is arbitrary, 
this yields $u \in H^2_{\rm loc} (B_R(0))$. 
Here $R>0$ is also arbitrary and we deduce $u \in H^2_{\rm loc} (\RN)$.

Finally, by $u \in H^2_{\rm loc} (\RN)$ with $u \in C^1(\RN)$, 
$A(\nabla u) \in H^1_{\rm loc} (\RN)$ and 
\[
\diver A(\nabla u) = \sum_{i,j=1}^N \frac{\partial A_i}{\partial \xi_j} (\nabla u) \partial_i \partial_j u.
\]
Recalling \eqref{deri-A}, we set 
\[
\wt{a}_{ij} (x) \coloneq \frac{\partial A_i}{\partial \xi_j} (\nabla u(x)) 
= \ab( 1 + \vab{\nabla u(x)}^{q-2} ) \delta_{ij} + (q-2) \vab{\nabla u(x)}^{q-2} \frac{\partial_i u(x)}{\vab{\nabla u(x)}} \frac{\partial_ju(x)}{\vab{\nabla u(x)}} \in C(\RN). 
\]
Since $u \in H^{2}_{\rm loc} (\RN)$ is a strong solution to 
\[
 - \sum_{i,j=1}^N \wt{a}_{ij} (x) \partial_i \partial_j u = f_u  \quad \text{in} \ \RN, \quad f_u \in C^1(\RN),
\]
the $W^{2,p}$-estimate for the uniform elliptic equations (\cite[Chapter 9]{GiTr01}, \cite[Chapter 3]{ChWu91}) 
yields $u \in W^{2,r}_{\rm loc} (\RN)$ for any $r \in (1,\infty)$, and hence $u \in C^{1,\beta}_{\rm loc} (\RN)$ for each $\beta \in (0,1)$. 
By $(DA)_{ij} \in C^\gamma_{\rm loc}(\RN,\RN)$ where $0<\gamma < \min \set{ q-2,1 }$, 
we have $\wt{a}_{ij} \in C^{\gamma}_{\rm loc}(\RN)$. 
Thus, the Schauder estimate gives $u \in C^{2,\gamma}_{\rm loc} (\RN)$ and this completes proof. 
\end{proof}

\begin{remark}
	The above argument works also for weak solutions in $\wh{X}$ and 
	the same conclusion to \cref{l:reg} holds for weak solutions in $\wh{X}$. 
\end{remark}

Since we know that every weak solution in $X$ becomes a classical solution, 
it is standard to obtain the Pohozaev identity by following the arguments in \cite[Proposition 1]{BeLi83}:

\begin{lemma}\label{lem_Poho}
	Let $N \geq 1$, $2 < q < \infty$, $p \in (2,q^*)$, $\lambda \in \R$, $\alpha > 0$ and 
	$u \in X$ be a weak solution of \eqref{eq_main}. Then $u$ satisfies the Pohozaev identity: 
	\begin{equation*}\label{eq_Poho}
		0 = \frac{N-2}{2} \Vab{\nabla u}_2^2 + \frac{N-q}{q} \Vab{\nabla u}_q^q + N \frac{\lambda}{2} \Vab{u}_2^2 - N \frac{\alpha}{p} \Vab{u}_p^p 
		\eqcolon P_{\alpha,\lambda} (u). 
	\end{equation*}
	The same is true for a weak solution $u \in \wh{X}$ when $N \geq 3$, $\lambda = 0$ and $p \in [2^*, q^*)$. 
\end{lemma}

\end{appendices}

\section*{Acknowledgments}
L.B. is a member of the {\em Gruppo Nazionale per l'Analisi Ma\-te\-ma\-ti\-ca, la Probabilit\`a e le loro Applicazioni} (GNAMPA) of the {\em Istituto Nazionale di Alta Matematica} (INdAM). 
Most of the paper was written while L.B. was affiliated with the University of Granada. L. B. was partially supported by the ``Maria de Maeztu'' Excellence Unit IMAG, reference CEX2020-001105-M, funded by MCIN/AEI/10.13039/501100011033/, by the Deutsche Forschungsgemeinschaft (DFG, German Research Foundation) - Project-ID 258734477 - SFB 1173 and by the INdAM-GNAMPA Project 2024 titled {\em Regolarità ed esistenza per operatori anisotropi} (E5324001950001). 
This work was supported by JSPS KAKENHI Grant Number JP24K06802.

\end{document}